\documentclass[11pt]{article}
\usepackage{amsmath,amsthm,amssymb}
\usepackage[usenames,dvipsnames]{xcolor}
\usepackage{enumerate}
\usepackage{graphicx}
\usepackage{cite}
\usepackage{comment}
\usepackage{oands}
\usepackage{tikz}
\usepackage{changepage}
\usepackage{bbm}
\usepackage{mathtools}
\usepackage[margin=1in]{geometry}
\usepackage[pagewise,mathlines]{lineno}
\usepackage{appendix}
\usepackage{stmaryrd} 
\usepackage{multicol}
\usepackage{microtype}
\usepackage[colorinlistoftodos]{todonotes}

\usepackage[pdftitle={KPZ formulas for the Liouville quantum gravity metric},
  pdfauthor={Ewain Gwynne and Josh Pfeffer},
colorlinks=true,linkcolor=NavyBlue,urlcolor=RoyalBlue,citecolor=PineGreen,bookmarks=true,bookmarksopen=true,bookmarksopenlevel=2,unicode=true,linktocpage]{hyperref}

\theoremstyle{plain}
\newtheorem{thm}{Theorem}[section]
\newtheorem{cor}[thm]{Corollary}
\newtheorem{lem}[thm]{Lemma}
\newtheorem{prop}[thm]{Proposition}

\def\@rst #1 #2other{#1}
\newcommand\MR[1]{\relax\ifhmode\unskip\spacefactor3000 \space\fi
  \MRhref{\expandafter\@rst #1 other}{#1}}
\newcommand{\MRhref}[2]{\href{http://www.ams.org/mathscinet-getitem?mr=#1}{MR#2}}

\theoremstyle{definition}
\newtheorem{defn}[thm]{Definition}
\newtheorem{notation-no-italics}[thm]{Notation}
\newtheorem{remark}[thm]{Remark}

\numberwithin{equation}{section}

\newcommand{\dsb}{\begin{adjustwidth}{2.5em}{0pt}
\begin{footnotesize}}
\newcommand{\dse}{\end{footnotesize}
\end{adjustwidth}}

\newcommand{\ssb}{\begin{adjustwidth}{2.5em}{0pt}}
\newcommand{\sse}{\end{adjustwidth}}

\newcommand{\aryb}{\begin{eqnarray*}}
\newcommand{\arye}{\end{eqnarray*}}
\def\alb#1\ale{\begin{align*}#1\end{align*}}
\def\allb#1\alle{\begin{align}#1\end{align}}
\newcommand{\eqb}{\begin{equation}}
\newcommand{\eqe}{\end{equation}}
\newcommand{\eqbn}{\begin{equation*}}
\newcommand{\eqen}{\end{equation*}}

\newcommand{\BB}{\mathbbm}
\newcommand{\ol}{\overline}

\newcommand{\op}{\operatorname}

\newcommand{\frk}{\mathfrak}
\newcommand{\eqD}{\overset{d}{=}}
\newcommand{\ep}{\varepsilon}
\newcommand{\rta}{\rightarrow}

\newcommand{\wt}{\widetilde}
 
\newcommand{\mcl}{\mathcal}

\newcommand{\bdy}{\partial}

\let\originalleft\left
\let\originalright\right
\renewcommand{\left}{\mathopen{}\mathclose\bgroup\originalleft}
\renewcommand{\right}{\aftergroup\egroup\originalright}

\title{KPZ formulas for the Liouville quantum gravity metric}

\date{ }
\author{
\begin{tabular}{c} Ewain Gwynne\\[-5pt]\small Cambridge \end{tabular}  
\begin{tabular}{c} Joshua Pfeffer\\[-5pt]\small MIT \end{tabular} 
}

\begin{document}

\maketitle


\begin{abstract}
Let $\gamma\in (0,2)$, let $h$ be the planar Gaussian free field, and let $D_h$ be the associated $\gamma$-Liouville quantum gravity (LQG) metric. 
We prove that for any random Borel set $X \subset \mathbb{C}$ which is independent from $h$, the Hausdorff dimensions of $X$ with respect to the Euclidean metric and with respect to the $\gamma$-LQG metric $D_h$ are a.s.\ related by the (geometric) KPZ formula.     
As a corollary, we deduce that the Hausdorff dimension of the continuum $\gamma$-LQG metric is equal to the exponent $d_\gamma > 2$ studied by Ding and Gwynne (2018), which describes distances in discrete approximations of $\gamma$-LQG such as random planar maps.

We also derive ``worst-case" bounds relating the Euclidean and $\gamma$-LQG dimensions of $X$ when $X$ and $h$ are not necessarily independent, which answers a question posed by Aru (2015). 
Using these bounds, we obtain an upper bound for the Euclidean Hausdorff dimension of a $\gamma$-LQG geodesic which equals $1.312\dots$ when $\gamma = \sqrt{8/3}$; and an upper bound of $1.9428\dots$ for the Euclidean Hausdorff dimension of a connected component of the boundary of a $\sqrt{8/3}$-LQG metric ball. 

We use the axiomatic definition of the $\gamma$-LQG metric, so the paper can be understood by readers with minimal background knowledge beyond a basic level of familiarity with the Gaussian free field. 
\end{abstract}

\tableofcontents

\section{Introduction}
\label{sec-intro}

\subsection{Overview}
\label{sec-overview}

Since the 1980s, physicists and mathematicians have studied a family of random surfaces known as $\gamma$-Liouville quantum gravity (LQG) surfaces with parameter $\gamma \in (0,2)$.
Heuristically, a $\gamma$-LQG surface is the random two-dimensional Riemannian manifold parametrized by a planar domain $U\subset \BB C$ with Riemannian metric tensor
\eqb \label{eqn-lqg-def}
e^{\gamma h} (dx^2 + dy^2),
\eqe
where $h$ is the Gaussian free field (GFF) on $U$, or some minor variant thereof, and $dx^2 + dy^2$ denotes the Euclidean metric tensor.  
Since the GFF is a distribution (generalized function), not a function, the formula~\eqref{eqn-lqg-def} does not make literal sense and various regularization procedures are needed to define LQG surfaces rigorously.

LQG surfaces were first introduced by Polyakov~\cite{polyakov-qg1,polyakov-qg2} as a canonical model of random two-dimensional Riemannian manifolds. Such surfaces arise as the scaling limits of random planar maps, one of the most natural discrete random surface models.
The case when $\gamma=\sqrt{8/3}$ (``pure gravity") corresponds to uniform random planar maps. Other values of $\gamma$ (``gravity coupled to matter") correspond to random planar maps sampled  with probability proportional to the partition function of a statistical mechanics model on the map, such as the uniform spanning tree ($\gamma = \sqrt{2}$) or the Ising model ($\gamma = \sqrt{3}$).
LQG surfaces are also closely related to Liouville conformal field theory (LCFT); see~\cite{dkrv-lqg-sphere,grv-higher-genus} for rigorous works on LCFT.

An early breakthrough in the study of LQG surfaces in the physics literature --- and still one of the key tools for understanding their geometry --- is the (geometric) \emph{Knizhnik-Polyakov-Zamolodchikov (KPZ) formula}.\footnote{The original KPZ formula in~\cite{kpz-scaling} described what the primary fields of the matter field CFT become when they are coupled to quantum gravity. A rigorous proof of this formula seems to be mathematically out of reach so far. In this paper, we instead consider a weaker formulation which relates the fractal dimension of a set sampled independently from the GFF as measured with the Euclidean metric to the fractal dimension of the same set as measured by the
random distance function corresponding to~\eqref{eqn-lqg-def}. This weaker version also goes under the name KPZ formula and is sometimes called the geometric KPZ formula. The first rigorous versions of the weaker formulation appeared in~\cite{shef-kpz,rhodes-vargas-log-kpz}. 
In this paper we use the terms KPZ formula and geometric KPZ formula interchangeably and we are only concerned with the formula that relates the two notions of dimension for
random fractals.
}
This formula relates the Euclidean Hausdorff (or Minkowski) dimension of a random fractal set $X\subset U$ sampled independently from $h$ to its ``quantum dimension", i.e., its dimension with respect to the $\gamma$-LQG metric tensor~\eqref{eqn-lqg-def}. 
If the quantum dimension is normalized to lie in $[0,1]$, then the KPZ formula reads
\eqb \label{eqn-kpz}
\Delta_0  = \left(2 + \frac{\gamma^2}{2} \right) \Delta_\gamma - \frac{\gamma^2}{2} \Delta_\gamma^2 ,
\eqe
where $\Delta_0 \in [0,2]$ and $\Delta_\gamma \in [0,1]$ denote the Euclidean and quantum dimensions of $X$, respectively. 
Historically, the KPZ formula has been an important tool for computing dimensions and exponents associated with random fractals, both rigorously and non-rigorously. 
For example, it was used by Duplantier to derive non-rigorously the values of the so-called \emph{Brownian intersection exponents}.\footnote{The Brownian intersection exponents were derived earlier using a different method by Duplantier and Kwon~\cite{duplantier-kwon-brownian}.} These values were later verified mathematically by Lawler, Schramm, and Werner in \cite{lsw-bm-exponents1,lsw-bm-exponents2,lsw-bm-exponents3} using Schramm-Loewner evolution (SLE) techniques.
As another example, the paper~\cite{ghm-kpz} uses a rigorous version of the KPZ formula to compute the dimensions of several sets associated with SLE.

An important difficulty in making rigorous sense of the KPZ formula is defining what is meant by ``quantum dimension''.
This is non-trivial since \emph{a priori}~\eqref{eqn-lqg-def} is ill-defined.
The most natural approach would be to define a $\gamma$-LQG surface as a metric space, i.e., to construct a metric $D_h$ on $U$ corresponding to~\eqref{eqn-lqg-def} and define the quantum dimension of $X$ to be the Hausdorff dimension of the metric space $(X,D_h|_X)$ (if we want the quantum dimension to lie in $[0,1]$, as in~\eqref{eqn-kpz}, we can divide by the Hausdorff dimension of the whole space $(U,D_h)$).  
However, until very recently it was not known how to rigorously construct such a metric $D_h$.
Consequently, the rigorous versions of the KPZ relation in previous literature use different geometric objects that can be constructed in the LQG context to define some reasonable notion of quantum dimension.

The first rigorous versions of the KPZ formula were proven by Duplantier and Sheffield~\cite{shef-kpz} and Rhodes and Vargas~\cite{rhodes-vargas-log-kpz}.
Both of these papers defined the quantum dimension in terms of the \emph{$\gamma$-LQG area measure} $\mu_h$, a random measure on $U$ which can be defined as a limit of regularized versions of $e^{\gamma h} \,dx\,dy$, where $dx\,dy$ denotes Lebesgue measure. See~\cite{kahane,rhodes-vargas-review} for a more general theory of random measures of this type, called \emph{Gaussian multiplicative chaos}; the result of~\cite{rhodes-vargas-log-kpz} in fact applies at this level of generality. 
In~\cite{shef-kpz}, the authors define the ``quantum dimension" of $X$ as (roughly speaking) the scaling exponent for the number of dyadic squares with $\gamma$-LQG mass approximately $\ep$ which intersect $X$, and showed that this exponent is related to the Euclidean Minkowski dimension of $X$ via the KPZ formula.
The authors of~\cite{rhodes-vargas-log-kpz} prove a similar result but for Hausdorff dimension instead of Minkowski dimension (the ``quantum Hausdorff dimension" is defined using the $\mu_h$-masses of Euclidean balls). 
Many other versions of the KPZ formula have since been established using different notions of quantum dimension~\cite{shef-kpz,aru-kpz,ghs-dist-exponent,ghm-kpz,grv-kpz,gwynne-miller-char,bjrv-gmt-duality,benjamini-schramm-cascades,wedges,shef-renormalization,shef-kpz}. 

The recent works~\cite{dddf-lfpp,local-metrics,lqg-metric-estimates,gm-confluence,gm-uniqueness,gm-coord-change} rigorously constructed the Riemannian distance function associated with~\eqref{eqn-lqg-def} for each $\gamma \in (0,2)$.
This distance function is a random metric on the planar domain $U$ which induces the same topology as the Euclidean metric.\footnote{Here and throughout the paper by ``metric" we mean in the sense of ``metric space" rather than ``Riemannian metric". There will be no risk of confusion since our results and proofs do not make any direct reference to Riemannian metrics.}
We will review the definition of this metric in Section~\ref{sec-metric-def}.  
The metric for $\gamma=\sqrt{8/3}$ was previously constructed (in an entirely different manner) by Miller and Sheffield~\cite{lqg-tbm1,lqg-tbm2,lqg-tbm3}, who also proved that a certain special $\sqrt{8/3}$-LQG surface --- the so-called \emph{quantum sphere} --- is equivalent, as a metric space, to the \emph{Brownian map}~\cite{legall-uniqueness,miermont-brownian-map}. 

Hence, it is now possible to state the most natural formulation of the KPZ formula, i.e., a relation between the Hausdorff dimensions of a subset of the complex plane measured with respect to the Euclidean metric and the $\gamma$-LQG metric. This formula is the first main result of this paper (Theorem~\ref{thm-kpz-independent}).

As a corollary (Corollary~\ref{cor-lqg-dim}), we prove that the Hausdorff dimension of the continuum $\gamma$-LQG metric is given by the LQG dimension exponent $d_\gamma > 2$ from~\cite{dg-lqg-dim} (see also~\cite{dzz-heat-kernel,ghs-map-dist}). The exponent $d_\gamma$ describes several different exponents related to discrete approximations of $\gamma$-LQG. One possible definition of $d_\gamma$ is as the ball volume growth exponent for certain random planar maps in the $\gamma$-LQG universality class. More precisely, it is shown in~\cite[Theorem 1.6]{dg-lqg-dim}, building on~\cite{dzz-heat-kernel,ghs-map-dist}, that for certain infinite-volume random planar maps $(M,v)$ which are expected to converge to $\gamma$-LQG in the scaling limit (including the uniform infinite planar triangulation~\cite{angel-schramm-uipt}, infinite spanning-tree weighted maps~\cite{shef-kpz,chen-fk}, and mated-CRT maps~\cite{gms-tutte}), the limit 
\eqb
d_\gamma := \lim_{r \rta\infty} \frac{\log |\mcl B_r(v)|}{\log r} 
\eqe 
exists and depends only on $\gamma$, where $|\mcl B_r(v)|$ is the number of vertices in the $M$-graph distance ball of radius $r$ centered at $v$. 
There are several other exponents which can be described in terms of $d_\gamma$, which lead to equivalent definitions of $d_\gamma$.
For example, the exponent for the graph-distance speed of random walk on the same random planar maps discussed above is $1/d_\gamma$~\cite{gm-spec-dim,gh-displacement} and the distance exponents for Liouville first passage percolation and Liouville graph distance are $1-(\gamma/d_\gamma)(2/\gamma+\gamma/2)$ and $1/d_\gamma$, respectively~\cite{dg-lqg-dim,dzz-heat-kernel}. 

The exponent $d_{\gamma}$ was proposed in~\cite{dg-lqg-dim} as a reasonable interpretation of the notion of ``dimension of $\gamma$-LQG'' that has been heuristically studied in the physics literature, e.g., by Watabiki~\cite{watabiki-lqg}, who predicted the value of this dimension as a function of $\gamma$. Our corollary rigorously justifies this interpretation in the context of the continuum $\gamma$-LQG metric. We note that Watabiki's prediction is known to be false at least for small $\gamma$~\cite{ding-goswami-watabiki}. It is a major open problem to determine the value of $d_\gamma$ except when $\gamma=\sqrt{8/3}$, in which case it is known that $d_{\sqrt{8/3}}=4$. See~\cite{dg-lqg-dim,gp-lfpp-bounds,ang-discrete-lfpp} for upper and lower bounds on $d_\gamma$ and~\cite{dg-lqg-dim,gp-lfpp-bounds} for discussion about its possible value.   
 
Our next set of results considers the more general context in which the fractal being considered is not necessarily independent from the field.  
Recall that the KPZ formula only applies when the random fractal $X$ is sampled independently from $h$. 
However, there are many random fractals which are \emph{not} independent from $X$ which one might be interested in --- such as $\gamma$-LQG geodesics and the boundaries of $\gamma$-LQG metric balls. 
To allow us to say something about such fractals, we will also prove a ``worst-case" variant of the KPZ formula which gives bounds relating the Euclidean and quantum dimensions of $X$ for an arbitrary random Borel set $X$ (Theorem~\ref{thm-kpz-correlated}). 
This answers a question posed by Aru~\cite{aru-kpz}.
As we will explain below, our bounds are the best possible without additional assumptions on $X$. 
As applications of our bounds, we prove an upper bound for the Euclidean dimension of $\gamma$-LQG geodesics and for the outer boundary of a $\sqrt{8/3}$-LQG metric ball (Corollaries~\ref{cor-geodesic-bound} and~\ref{cor-ball-bdy}). 

Before stating our main results, we will review the axioms from~\cite{gm-uniqueness,gm-coord-change} that uniquely characterize the LQG metric.  
Because the LQG metric is defined axiomatically, our proofs do not require any outside input except some results from~\cite{ghm-kpz} and~\cite{lqg-metric-estimates}. All of these results are themselves proved using only basic properties of the GFF, and we re-state them as needed.  To understand our paper, the reader needs only to be familiar with basic properties of the GFF, as reviewed, e.g., in~\cite{shef-gff}, the introductory sections of~\cite{ss-contour,ig1,ig4}, or the notes~\cite{berestycki-lqg-notes}. 
\medskip


\noindent
\textbf{Acknowledgments.} We thank an anonymous referee for helpful comments on an earlier version of this paper.
Part of the project was carried out during J.\ Pfeffer's visit to the Isaac Newton Institute in Cambridge, UK in Summer 2018. We thank the institute for its hospitality.  
E.\ Gwynne was supported by a Herchel Smith fellowship and a Trinity College junior research fellowship. 
J.\ Pfeffer was partially supported by the National Science Foundation Graduate Research Fellowship under Grant No. 1122374.  

\subsection{The definition of the LQG metric}
\label{sec-metric-def}

It is shown in~\cite{gm-uniqueness} that for each $\gamma \in (0,2)$, there is a unique metric associated with $\gamma$-LQG which satisfies a certain list of axioms.
This metric can be constructed explicitly as the limit of appropriate regularizations --- defined by exponentiating a mollified version of the GFF~\cite{dddf-lfpp,lqg-metric-estimates,gm-uniqueness} --- but we will only need the axiomatic definition here. 
In order to make our results as general as possible we will state the axioms for the $\gamma$-LQG metric on a general domain $U\subset\BB C$, which were established in~\cite{gm-coord-change} (the paper~\cite{gm-uniqueness} only gives the list of axioms in the whole-plane case). 
In order to state these axioms we will need some elementary metric space definitions.

\begin{defn}
Let $(X,D)$ be a metric space.
\begin{itemize}
\item
For a curve $P : [a,b] \rta X$, the \emph{$D$-length} of $P$ is defined by 
\eqbn
\op{len}\left( P ; D  \right) := \sup_{T} \sum_{i=1}^{\# T} D(P(t_i) , P(t_{i-1})) 
\eqen
where the supremum is over all partitions $T : a= t_0 < \dots < t_{\# T} = b$ of $[a,b]$. Note that the $D$-length of a curve may be infinite.
\item
We say that $(X,D)$ is a \emph{length space} if for each $x,y\in X$ and each $\ep > 0$, there exists a curve of $D$-length at most $D(x,y) + \ep$ from $x$ to $y$. 
\item
For $Y\subset X$, the \emph{internal metric of $D$ on $Y$} is defined by
\eqb \label{eqn-internal-def}
D(x,y ; Y)  := \inf_{P \subset Y} \op{len}\left(P ; D \right) ,\quad \forall x,y\in Y 
\eqe 
where the infimum is over all paths $P$ in $Y$ from $x$ to $y$. 
Note that $D(\cdot,\cdot ; Y)$ is a metric on $Y$, except that it is allowed to take infinite values.  
\item
If $X$ is an open subset of $\BB C$, we say that $D$ is  a \emph{continuous metric} if it induces the Euclidean topology on $X$. 
We equip the set of continuous metrics on $X$ with the local uniform topology on $X\times X$ and the associated Borel $\sigma$-algebra.
\end{itemize}
\end{defn}

\medskip
\noindent
The defining properties of the $\gamma$-LQG metric involve the parameters
\eqb
Q = Q_\gamma =  \frac{2}{\gamma} + \frac{\gamma}{2} \quad \text{and} \quad \xi = \xi_\gamma := \frac{\gamma}{d_\gamma} , \label{eqn-xi-def}
\eqe 
where $d_\gamma > 2$ is the LQG dimension exponent from~\cite{dg-lqg-dim,dzz-heat-kernel,ghs-map-dist}, as mentioned above. We will show in this paper (Corollary~\ref{cor-lqg-dim}) that $d_\gamma$ is the Hausdorff dimension of the $\gamma$-LQG metric.
The parameter $Q$ appears in the conformal coordinate change formula for LQG. 
The relevance of the parameter $\xi$ is that the LQG metric is in some sense constructed from $e^{\xi h}$ instead of $e^{\gamma h}$; see, e.g.,~\cite{dddf-lfpp,lqg-metric-estimates,gm-uniqueness}. This was in fact pointed out in earlier work~\cite[Footnote 2]{ding-goswami-watabiki}. 
\medskip

\begin{defn}[The LQG metric]
\label{def-metric}
For an open set $U\subset \BB C$, let $\mcl D'(U)$ be the space of distributions (generalized functions) on $\BB C$, equipped with the usual weak topology.   
A \emph{$\gamma$-LQG metric} is a collection of measurable functions $h\mapsto D_h$, one for each open set $U\subset\BB C$, from $\mcl D'(U)$ to the space of continuous metrics on $U$ with the following properties. 
Let $U\subset \BB C$ be open and let $h$ be a \emph{GFF plus a continuous function} on $U$: i.e., $h$ is a random distribution on $U$ which can be coupled with a random continuous function $f$ in such a way that $h-f$ has the law of the (zero-boundary or whole-plane, as appropriate) GFF on $U$.  Then the associated metric $D_h$ satisfies the following axioms, where here we take $U$ to be fixed when we say ``almost surely".
\begin{enumerate}[I.]
\item \textbf{Length space.} Almost surely, $(U,D_h)$ is a length space, i.e., the $D_h$-distance between any two points of $U$ is the infimum of the $D_h$-lengths of $D_h$-continuous paths (equivalently, Euclidean continuous paths) in $U$ between the two points. \label{item-metric-length}
\item \textbf{Locality.} Let $V \subset U$ be a deterministic open set. 
The $D_h$-internal metric $D_h(\cdot,\cdot ; V)$ is a.s.\ equal to $D_{h|_V}$, so in particular it is a.s.\ given by a measurable function of $h|_V$.  \label{item-metric-local}
\item \textbf{Weyl scaling.} With $\xi$ defined in~\eqref{eqn-xi-def} and for a continuous function $f : U\rta \BB R$, define
\eqb \label{eqn-metric-f}
(e^{\xi f} \cdot D_h) (z,w) := \inf_{P : z\rta w} \int_0^{\op{len}(P ; D_h)} e^{\xi f(P(t))} \,dt , \quad \forall z,w\in U,
\eqe 
where the infimum is over all continuous paths from $z$ to $w$ in $U$ parametrized by $D_h$-length.
Then a.s.\ $ e^{\xi f} \cdot D_h = D_{h+f}$ for every continuous function $f: U\rta \BB R$. \label{item-metric-f}
\item \textbf{Conformal coordinate change.} Let $\wt U\subset \BB C$ and let $\phi : U \rta \wt U$ be a deterministic conformal map. Then, with $Q$ defined in~\eqref{eqn-xi-def}, a.s.\ \label{item-metric-coord}
\eqb \label{eqn-metric-coord0}
 D_h \left( z,w \right) = D_{h\circ\phi^{-1} + Q\log |(\phi^{-1})'|}\left(\phi(z) , \phi(w) \right)  ,\quad  \forall z,w \in U.
\eqe    
\end{enumerate}
\end{defn}

From~\cite{dddf-lfpp,local-metrics,lqg-metric-estimates,gm-confluence,gm-uniqueness,gm-coord-change}, we have the following existence and uniqueness result.

\begin{thm} \label{thm-metric}
For each $\gamma \in (0,2)$, a $\gamma$-LQG metric exists and is unique up to a deterministic global multiplicative scaling factor.
\end{thm}

More precisely, it is shown in~\cite[Theorem 1.2]{gm-uniqueness}, building on~\cite{dddf-lfpp,local-metrics,lqg-metric-estimates,gm-confluence}, that for each $\gamma \in (0,2)$, there is a unique (up to a deterministic global multiplicative constant) function $h\mapsto D_h$ from $\mcl D'(\BB C)$ to the space of continuous metrics on $\BB C$ which satisfies the conditions of Definition~\ref{def-metric} for $U=\BB C$ (note that this means $\phi$ in Axiom~\ref{item-metric-coord} is required to be a complex affine map). As explained in~\cite[Remark 1.5]{gm-uniqueness}, this gives a way to define $D_h$ whenever $h$ is a GFF plus a continuous function on an open domain $U\subset\BB C$ in such a way that Axioms~\ref{item-metric-length} through~\ref{item-metric-f} hold.
Note that the metric in the whole-plane case determines the metric on other domains due to Axiom~\ref{item-metric-local}.
It is shown in~\cite[Theorem 1.1]{gm-coord-change} that with the above definition, Axiom~\ref{item-metric-coord} holds for general conformal maps. 
 
Because of Theorem~\ref{thm-metric}, we may refer to the unique metric satisfying Definition~\ref{def-metric} as \emph{the $\gamma$-LQG metric}. Technically, the metric is unique only up to a global deterministic multiplicative constant.  When referring to the $\gamma$-LQG metric, we are implicitly fixing the constant in some arbitrary way. For example, we could require that the median distance between the left and right sides of $[0,1]^2$ is 1 when $h$ is a whole-plane GFF normalized so that its average over the unit circle is zero. The choice of constant will not play any role in our results or proofs. 

\subsection{Main results}

Recall that for $\Delta > 0$, the $\Delta$-\emph{Hausdorff content} of a metric space $(X,D)$ is the number
\eqb
C^\Delta(X,D) := \inf\left\{\sum_{j=1}^\infty r_j^\Delta : \text{there is a cover of $X$ by $D$-balls of radii $\{r_j\}_{j\geq 0}$} \right\} 
\eqe
and the \emph{Hausdorff dimension} of $(X,D)$ is defined to be $\inf\{\Delta > 0: C^\Delta(X,D) = 0\}$. 

Let $h$ be a GFF plus a continuous function on an open domain $U \subset \BB{C}$ (as in Definition~\ref{def-metric}), and let $D_h$ be the associated $\gamma$-LQG metric on $U$. For a set $X \subset \BB{C}$, let $\dim_{\mcl H}^0  X$ and $\dim_{\mcl H}^{\gamma}  X$ denote the Hausdorff dimensions of the metric spaces $(X , |\cdot|)$ and $(X,D_h)$, respectively.  We will refer to these quantities as the \textit{Euclidean dimension} and \textit{quantum dimension} of the set $X$, respectively. Note that the quantum dimension, like $D_h$, depends on the choice of parameter $\gamma$.

With $\xi$ and $Q$ as in~\eqref{eqn-xi-def}, we have the following KPZ relation between the Euclidean and quantum dimensions of a deterministic Borel set $X$---or, equivalently, of a random Borel set $X$ that is independent from $h$. 
 
\begin{thm}
Let $X \subset U$ be a deterministic Borel set or a random Borel set independent from $h$.  Then a.s.\
\eqb
\dim_{\mcl H}^0 X =  \xi Q \dim_{\mcl H}^{\gamma}  X  - \frac{\xi^2}{2}  (\dim_{\mcl H}^{\gamma}  X)^2  .
\label{kpz-independent}
\eqe
\label{thm-kpz-independent}
\end{thm}

Observe that~\eqref{kpz-independent} matches the KPZ formula~\eqref{eqn-kpz} if we define the Euclidean dimension $\Delta_0 = \dim_{\mcl H}^0 X$ and the re-scaled quantum dimension $\Delta_\gamma = \frac{1}{d_\gamma} \dim_{\mcl H}^\gamma X$. 
Since the Miller-Sheffield $\sqrt{8/3}$-LQG metric~\cite{lqg-tbm1,lqg-tbm2,lqg-tbm3} agrees with the $\sqrt{8/3}$-LQG metric of Definition~\ref{def-metric}~\cite[Corollary 1.4]{gm-uniqueness}, Theorem~\ref{thm-kpz-independent} answers in the affirmative the problem posed in \cite[Problem 9.5]{lqg-tbm2} (which asks for a proof of the KPZ formula for the $\sqrt{8/3}$-LQG metric).

Throughout this paper, for $z\in U$ and $\ep > 0$ such that $B_\ep(z)\subset U$, we write $h_\ep(z)$ for the average of the GFF $h$ over the circle $\bdy B_\ep(z)$; see~\cite[Section 3.1]{shef-kpz} for the basic properties of circle averages.
In the course of proving Theorem~\ref{thm-kpz-independent}, we also establish for each $\alpha \in [-2,2]$ a formula for the quantum dimension of the intersection of $X$ with the set $\mcl T_h^\alpha$ of \textit{$\alpha$-thick points} of the field $h$, defined in~\cite{hmp-thick-pts}  as
\eqb \label{eqn-thick-pts}
\mcl T_h^\alpha := \left\{ z\in U : \lim_{\ep\rta 0} \frac{h_\ep(z)}{\log\ep^{-1}} = \alpha \right\} .
\eqe 
It is known\footnote{The result~\cite[Theorem 4.1]{ghm-kpz} is stated for  a zero-boundary GFF; the statement for a whole-plane GFF follows from local absolute continuity.}~\cite[Theorem 4.1]{ghm-kpz} that if $X \subset U$ is a deterministic Borel set or a random Borel set independent from $h$, the Euclidean dimension of $X\cap \mcl T_h^\alpha$ is given by
\eqb
\label{eqn-thick-dim}
\dim_{\mcl H}^0\left( X\cap \mcl T_h^\alpha \right) = \max\left\{ \dim_{\mcl H}^0 X - \frac{\alpha^2}{2}  , 0 \right\} \qquad a.s.
\eqe 
We prove the analogue of this for quantum dimension. 

\begin{thm}
Let $X \subset U$ be a deterministic Borel set or a random Borel set independent from $h$.  Then, in the notation~\eqref{eqn-thick-pts}, for each $\alpha \in [-2,2]$ a.s.
\eqb
\dim_{\mcl H}^{\gamma}\left( X \cap \mcl T_h^\alpha  \right) = \frac{1}{\xi(Q - \alpha)}\max\left\{ \dim_{\mcl H}^0 X - \frac{\alpha^2}{2}  , 0 \right\} .
\label{kpz-independent-thick}
\eqe
\label{thm-kpz-independent-thick} 
\end{thm}

As an immediately corollary of Theorems~\ref{thm-kpz-independent} and~\ref{thm-kpz-independent-thick}, we have the following. 

\begin{cor} \label{cor-lqg-thick}
Let $X \subset U$ be a deterministic Borel set or a random Borel set independent from $h$.  Then, in the notation~\eqref{eqn-thick-pts}, a.s.\
\eqb
\dim_{\mcl H}^{\gamma}  X = \dim_{\mcl H}^{\gamma} \left( X\cap \mcl T_h^\alpha \right)  \quad \text{for} \quad \alpha = Q - \sqrt{Q^2 - 2 \dim_{\mcl H}^0 X}  .
\label{thickdim}
\eqe 
\end{cor}

By Corollary~\ref{cor-lqg-thick} applied with $X = U$, we get that the $\gamma$-LQG dimension exponent $d_\gamma$ from~\cite{dzz-heat-kernel,dg-lqg-dim}, as discussed in Section~\ref{sec-overview}, is the Hausdorff dimension of the $\gamma$-LQG metric.

\begin{cor} \label{cor-lqg-dim}
Almost surely, $\dim_{\mcl H}^{\gamma}(U) = \dim_{\mcl H}^{\gamma}\left(\mcl T_h^\gamma \right) = d_\gamma$. 
\end{cor}


The fact that the set of $\gamma$-thick points has full dimension in the setting of Corollary~\ref{cor-lqg-dim} is in some sense a metric analogue of the fact that the $\gamma$-LQG measure assigns full mass to $\mcl T_h^\gamma$~\cite{kahane}; see, e.g.,~\cite[Section 3.3]{shef-kpz} for a related discussion in English. 

Theorem~\ref{thm-kpz-independent} only applies when the set $X$ is independent from $h$. But there are also many sets which are \emph{not} independent from $h$ whose dimensions we might be interested in, for example $\gamma$-LQG geodesics, the boundaries of $\gamma$-LQG metric balls, and quantum Loewner evolution processes as considered in~\cite{qle}. 
To deal with such sets, we prove the following ``worst-case'' KPZ bounds relating $\dim_{\mcl H}^0 X$ and $\dim_{\mcl H}^{\gamma}  X$ which hold a.s.\ for \emph{every} Borel set $X\subset U$. 

\begin{thm}
Almost surely, for every Borel set $X\subset U$ one has
\eqb
\dim_{\mcl H}^{\gamma} X \leq  
\begin{dcases}
\frac{\dim_{\mcl H}^0 X  }{\xi \left( Q - \sqrt{  4  - 2 \dim_{\mcl H}^0 X    }\right)} ,\quad &\text{if} \: \dim_{\mcl H}^0 X < 2-\frac{\gamma^2}{2} \\
 d_\gamma ,\quad  &\text{if} \: \dim_{\mcl H}^0 X \geq  2-\frac{\gamma^2}{2}
 \end{dcases}
\label{upper-bound-quantum}
\eqe
and
\allb
& \nonumber \dim_{\mcl H}^0 X \\ & \leq 
\begin{dcases} 
 \xi \dim_{\mcl H}^{\gamma} X \left( Q - \xi \dim_{\mcl H}^{\gamma} X  + \sqrt{4 - 2Q \xi \dim_{\mcl H}^{\gamma} X  + \xi^2 (\dim_{\mcl H}^{\gamma} X)^2} \right),\:\: &\text{if} \: \dim_{\mcl H}^{\gamma} X < \frac{2}{\xi Q}   \\
2,\:\:  &\text{if} \: \dim_{\mcl H}^{\gamma} X \geq  \frac{2}{\xi Q} 
\end{dcases}
\label{upper-bound-euclidean}
\alle
\label{thm-kpz-correlated} 
\end{thm}

\begin{figure}[t!]
 \begin{center}
\includegraphics[scale=.6]{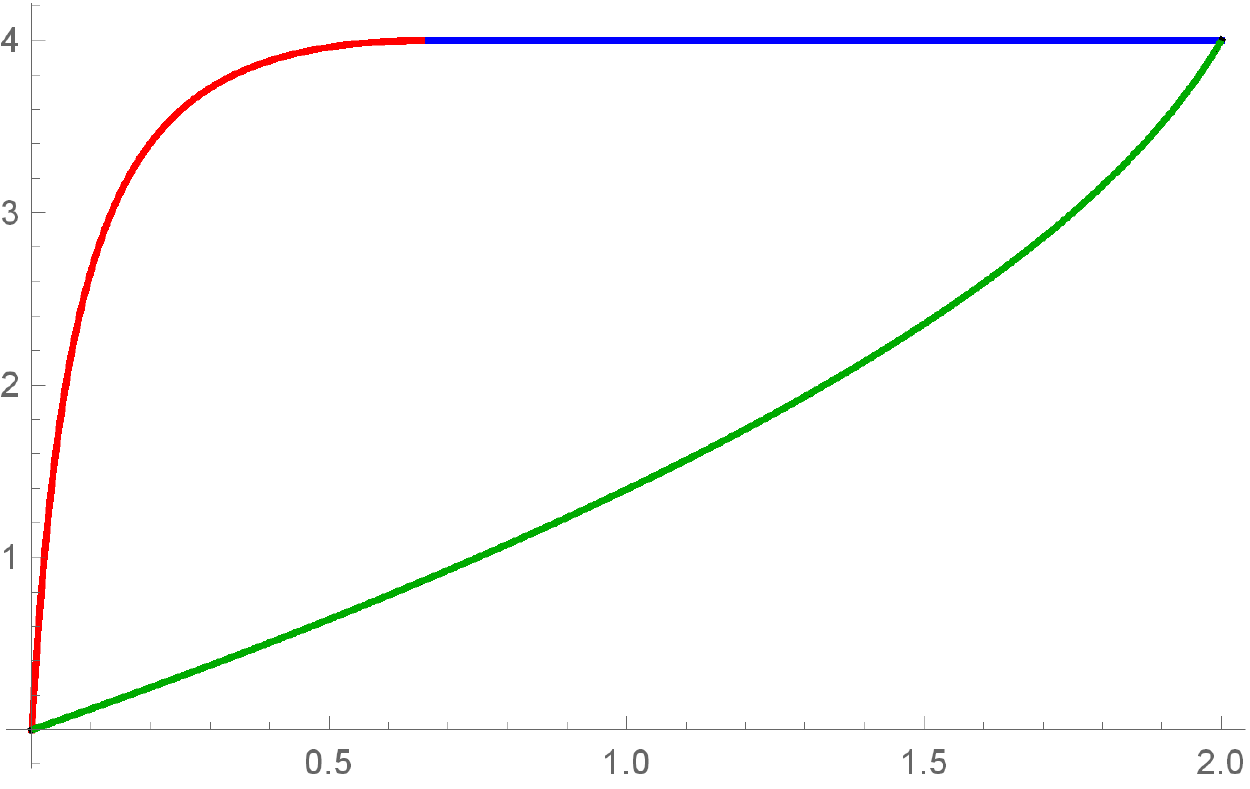} \hspace{15pt} \includegraphics[scale=.6]{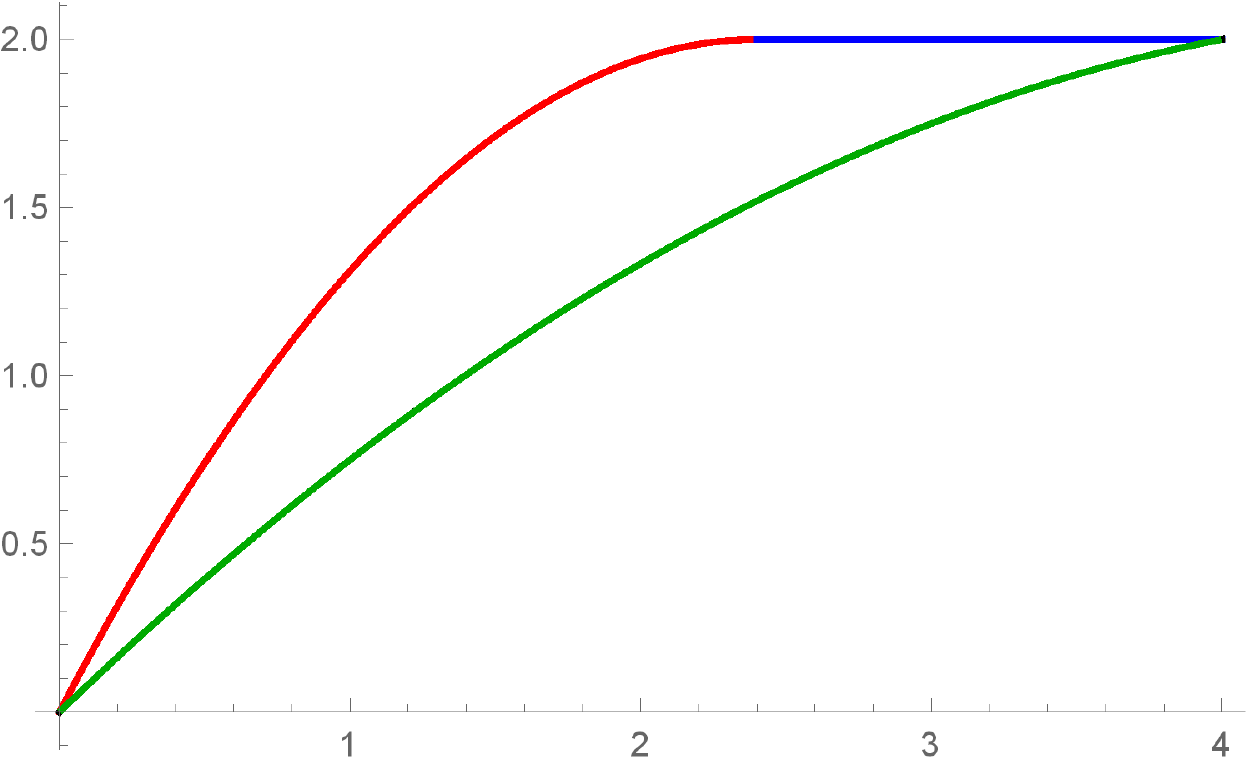}
\vspace{-0.01\textheight}
\caption{ Graphs of the bounds of Theorem~\ref{thm-kpz-correlated} in the case $\gamma = \sqrt{8/3}$ (for which $d_\gamma = 4$).  The graph on the left shows the upper bound~\eqref{upper-bound-quantum} for $\dim^\gamma_{\mcl H} X$ as a function of $\dim^0_{\mcl H} X$ (solid red and blue) and the KPZ relation between these two dimensions in the case when $X$ is independent from $h$ (dashed green). The graph on the right shows the upper bound~\eqref{upper-bound-euclidean} for $\dim^0_{\mcl H} X$ as a function of $\dim^\gamma_{\mcl H} X$ (solid red and blue) and the KPZ relation in the case when $X$ is independent from $h$ (dashed green). In both graphs, the segment where the slope of the upper curve is positive (colored in red) corresponds to the range of values for which the bound is nontrivial.}
\label{fig-kpz-correlated}
\end{center}
\vspace{-1em}
\end{figure} 
 
See Figure~\ref{fig-kpz-correlated} for a plot of the formulas from Theorems~\ref{thm-kpz-independent} and~\ref{thm-kpz-correlated}. 
The bound~\eqref{upper-bound-quantum} is optimal in the case when $X =\mcl T_h^\alpha$ is the set of $\alpha$-thick points for any $\alpha \in [\gamma,2]$.  
Indeed, by~\cite{hmp-thick-pts} (or~\eqref{eqn-thick-dim} with $X = U$)  and Corollary~\ref{cor-lqg-dim}, the upper bound for $\dim_{\mcl H}^\gamma X$ coincides with $\dim_{\mcl H}^\gamma (\mcl T_h^\alpha)$ in this case. Note that every value of $\dim_{\mcl H}^0 X$ for which the bound~\eqref{upper-bound-quantum} on $\dim_{\mcl H}^\gamma X$ is nontrivial is achieved by $X =\mcl T_h^\alpha$ for some $\alpha \in [\gamma,2]$.
Similarly, the bound~\eqref{upper-bound-euclidean} is optimal in the case when $X = \mcl T_h^\alpha$ for $\alpha \in [-2,0]$, and this covers the entire range of quantum dimensions for which the bound~\eqref{upper-bound-euclidean} is nontrivial.
Hence the bounds of Theorem~\ref{thm-kpz-correlated} cannot be improved anywhere on their respective domains without additional hypotheses on $X$. 

Theorem~\ref{thm-kpz-correlated} is proven via a ``worst-case" analysis whereby we start with a covering of $X$ by squares with a given Euclidean (resp.\ quantum) size and assume that the squares in our covering are the ones with the largest possible quantum (resp.\ Euclidean) sizes. We then deduce the theorem by bounding the number of squares with a given Euclidean and quantum size. See Section~\ref{sec-kpz-correlated} for details. 
 
KPZ-type relations for random sets which are not required to be independent from $h$ were previously considered by Aru~\cite{aru-kpz} (using a notion of ``quantum dimension" defined in terms of Euclidean balls of a given LQG area). 
He computed the quantum dimension of an SLE$_\kappa$ flow line of the GFF (in the sense of~\cite{ig1}) and observed that this dimension does not satisfy the KPZ relation. 
We expect that, in our notation, for a flow line $\eta$ one has that $\frac{1}{d_\gamma} \dim_{\mcl H}^\gamma \eta$ is given by the same formula as in~\cite{aru-kpz}.

Theorem~\ref{thm-kpz-correlated} solves a variant of~\cite[Question 7.4]{aru-kpz}, which asks for the optimal bounds relating Euclidean and quantum dimension for sets which are not independent from the GFF.

\begin{remark}
\label{remark-holder}
It is shown in~\cite[Theorem 1.7]{lqg-metric-estimates} that $D_h$ is a.s.\ locally bi-H\"older continuous with respect to the Euclidean metric, in the sense that the identity map from $U$, equipped with the Euclidean metric, to $(U , D_h)$ is locally H\"older continuous with any exponent less than $\xi(Q-2)$, and the inverse of this map is locally H\"older continuous with any exponent less than $\xi^{-1}(Q+2)^{-1}$. This implies that that a.s.\ for any Borel set $X\subset\BB C$, 
\eqb \label{eqn-holder-bounds}
 \frac{ \dim_{\mcl H}^0 X  }{\xi(Q+2)}  \leq  \dim_{\mcl H}^{\gamma} X \leq \frac{ \dim_{\mcl H}^0 X  }{\xi(Q-2)}  .
\eqe
The bounds of Theorem~\ref{thm-kpz-correlated} are strictly better than~\eqref{eqn-holder-bounds} except in degenerate cases. Hence Theorem~\ref{thm-kpz-correlated} provides more refined bounds relating the Euclidean metric and the LQG metric than the optimal H\"older exponents.
\end{remark}

The bound~\eqref{upper-bound-euclidean} is closest to optimal when $h$ is likely to be ``small" on $X$. 
An example of a set on which $h$ is likely to be small is a $\gamma$-LQG geodesic, since larger values of $h$ correspond to larger $\gamma$-LQG distances.
It was shown in~\cite{mq-geodesics} that $\gamma$-LQG geodesics have Euclidean dimension strictly less than 2. Using Theorem~\ref{thm-kpz-correlated}, we substantially improve on this bound.

\begin{cor}
\label{cor-geodesic-bound}
Almost surely, the Euclidean Hausdorff dimension of every $D_h$-geodesic is at most
\eqb \label{eqn-geodesic-bound}
\xi \left( Q - \xi + \sqrt{\xi^2 - 2 Q \xi + 4} \right) .
\eqe
In fact, it is a.s.\ the case that the Euclidean Hausdorff dimension of every $D_h$-rectifiable path is at most~\eqref{eqn-geodesic-bound}. 
\end{cor}
\begin{proof}
The $D_h$-Hausdorff dimension of a $D_h$-rectifiable curve is 1, so the result follows from~\eqref{upper-bound-euclidean}.
\end{proof}

The upper bound~\eqref{eqn-geodesic-bound} coincides with the upper bound for the ``Euclidean dimension" of Liouville first passage percolation geodesics obtained in~\cite[Corollary 2.7]{gp-lfpp-bounds}. This is to be expected since Theorem~\ref{thm-kpz-correlated} and~\cite[Corollary 2.7]{gp-lfpp-bounds} are proven via a similar ``worst-case" analysis. 
When $\gamma = \sqrt{8/3}$, we know that $d_{\sqrt{8/3}}=4$ (equivalently, $\xi =1/\sqrt 6$), so Corollary~\ref{cor-geodesic-bound} yields an upper bound of $\frac{4 + \sqrt{15}}{6} \approx 1.31216$ for the Euclidean Hausdorff dimension of $\sqrt{8/3}$-LQG geodesics. 
For general $\gamma\in (0,2)$, one can plug in the bounds for $d_\gamma$ from~\cite{dg-lqg-dim,gp-lfpp-bounds,ang-discrete-lfpp} to obtain an explicit non-trivial upper bound for the dimension of $\gamma$-LQG-geodesics.  

Another interesting random fractal associated with $D_h$ is the boundary of a $\gamma$-LQG metric ball.
Typically, this boundary is not connected since the metric ball has ``holes". 
Currently, the quantum dimension of a boundary component of a $\gamma$-LQG metric ball is not known except when $\gamma =\sqrt{8/3}$, and the Euclidean dimension of such a component is not known for any $\gamma \in (0,2)$.  
In the case when $\gamma=\sqrt{8/3}$, the complementary connected components of an LQG metric ball on the quantum sphere (equivalently, the Brownian map), viewed as metric spaces equipped with their internal metrics, are Brownian disks~\cite{lqg-tbm1,legall-disk-snake}.
Hence, it follows from~\cite[Theorem 3]{bet-disk-tight} that the quantum dimension of the boundary of each of these components is at most\footnote{We know from~\cite[Theorem 3]{bet-disk-tight} that the quantum dimension of the boundary of each component w.r.t.\ the $D_h$-internal metric on the component is $2$.  This implies that the dimensions of these boundaries  w.r.t.\ $D_h$ itself are at most $2$. It is expected that these dimensions are in fact equal to $2$ as well, but a proof of the lower bound has not been written down.} 2.
From this and Theorem~\ref{thm-kpz-correlated}, we get a non-trivial upper bound for the corresponding Euclidean dimension. 

\begin{cor} \label{cor-ball-bdy}
Let $\gamma=\sqrt{8/3}$ and fix $z \in \BB C$ and $s>0$. Almost surely, the Euclidean Hausdorff dimension of the boundary of each complementary connected component of the $\sqrt{8/3}$-LQG metric ball of radius $s$ centered at $z$ is at most $\frac13 (3 + 2 \sqrt 2) \approx 1.9428$.
\end{cor}

\begin{remark}
The more recent paper~\cite{gwynne-ball-bdy} computes the essential supremum of the Euclidean Hausdorff dimension of the \emph{whole} boundary (not just a boundary component) of an LQG metric ball for each $\gamma \in (0,2)$. The results of~\cite{gwynne-ball-bdy} show in particular that the Euclidean Hausdorff dimension of the boundary of a $\sqrt{8/3}$-LQG metric ball is a.s.\ at most $5/4$. This gives a better upper bound for the Euclidean dimension of a connected component of the boundary than the one of Corollary~\ref{cor-ball-bdy}. However, we still do not expect this bound is optimal since we expect that the Hausdorff dimension of the whole ball boundary is strictly larger than the Hausdorff dimension of a single connected component of the ball boundary. 
\end{remark}

Another potential application of Theorem~\ref{thm-kpz-correlated} involves the so-called \emph{quantum Loewner evolution (QLE)} processes constructed in~\cite{qle} by ``re-shuffling" Schramm-Loewner evolution (SLE$_\kappa$) curves on a $\gamma$-LQG surface for $\kappa \in \{\gamma^2,16/\gamma^2\}$. At least when $\kappa  = \gamma^2 \in (0,4)$, we expect that the $\gamma$-LQG dimension of a QLE is the same as the $\gamma$-LQG dimension of the corresponding SLE$_\kappa$ curve, which is $d_\gamma/2$. However, the Euclidean dimension of the QLE is unknown. Theorem~\ref{thm-kpz-independent} can be used to get a non-trivial upper bound for this Euclidean dimension. 

\subsection{Outline and basic notation}
\label{sec-outline}

Since part of Theorem~\ref{thm-kpz-independent} follows immediately from Theorem~\ref{thm-kpz-independent-thick}, we  present first the proof of Theorem~\ref{thm-kpz-independent-thick} in Section~\ref{sec-thm-kpz-independent-thick}, and then the proof of Theorem~\ref{thm-kpz-independent}  in Section~\ref{sec-kpz-independent}.
The proofs of these theorems are somewhat similar to arguments used in~\cite[Section 5]{ghm-kpz}, but we use estimates for the $\gamma$-LQG metric from~\cite{lqg-metric-estimates} instead of estimates for space-filling SLE.  
Finally, we prove Theorem~\ref{thm-kpz-correlated} in Section~\ref{sec-kpz-correlated}.
\medskip

\noindent  Throughout the paper, we use the following basic notation:
\medskip

\noindent
We write $\BB N = \{1,2,3,\dots\}$. 
\medskip

\noindent
For $a < b$, we define the discrete interval $[a,b]_{\BB Z}:= [a,b]\cap\BB Z$. 
\medskip

\noindent
If $f  :(0,\infty) \rta \BB R$ and $g : (0,\infty) \rta (0,\infty)$, we say that $f(\ep) = O_\ep(g(\ep))$ (resp.\ $f(\ep) = o_\ep(g(\ep))$) as $\ep\rta 0$ if $f(\ep)/g(\ep)$ remains bounded (resp.\ tends to zero) as $\ep\rta 0$. We similarly define $O(\cdot)$ and $o(\cdot)$ errors as a parameter goes to infinity. 
\medskip

\noindent
For a subset $X$ of $\BB{C}$, we let $B_r(X)$ denote the $r$-Euclidean neighborhood of the set $X$; i.e., the set of points in $\BB C$ which lie at Euclidean distance strictly less than $r$ from $X$.  In the case in which $X = \{x\}$ is a single point, we write $B_r(x) = B_r(\{x\})$.
\medskip
 

\section{Quantum dimension of the $\alpha$-thick points}
\label{sec-thm-kpz-independent-thick}

In this section we will prove Theorem~\ref{thm-kpz-independent-thick}.

\subsection{Estimates for the circle average process}
\label{sec-circle-avg-estimates} 

Let $h$ be a whole-plane GFF with the additive constant chosen so that $h_1(0) = 0$. 
To analyze thick points of $h$, we will need some results from~\cite{ghm-kpz} that we now describe.
Recall from~\eqref{eqn-thick-pts} that $h_\ep(z)$ denotes the average of $h$ over $\bdy B_\ep(z)$ and $\mcl T_h^\alpha$ denotes the set of $\alpha$-thick points of $h$.
For our purposes, we will need a set of points for which $h_\ep(z) / \log(\ep^{-1})$ is \emph{uniformly} close to $\alpha$ (note that the rate of convergence in~\eqref{eqn-thick-pts} can depend on $z$). The following lemma, which is very similar to the statement of~\cite[Lemma 4.3]{ghm-kpz} (and is proven in essentially the same way), asserts that we can find such a set whose Euclidean and quantum dimensions are not too much different from those of $X\cap \mcl T_h^\alpha$. 

\begin{lem}  \label{lem-thick-subset}
Let $\alpha \in \BB R$ and $\zeta  > 0$. Almost surely, for each bounded Borel set $X\subset \BB C$ there exists a random $ \overline{\ep} >0$ depending on $\alpha $, $\zeta$, and $X$ such that the following is true. If we set
\begin{equation}\label{lower_b2}
X^{\alpha, \zeta} := \left\{ z \in X: \frac{h_{\ep}(z)}{\log (\ep^{-1})} \in [\alpha - \zeta, \alpha+ \zeta], \forall \ep \in (0, \overline{\ep}] \right\}
\end{equation}
then $\dim_{\mcl{H}}^0 X^{\alpha, \zeta}  \geq \dim_{\mcl H}^0 (X \cap \mcl T_h^\alpha)- \zeta$ and $\dim_{\mcl{H}}^\gamma X^{\alpha, \zeta}  \geq \dim_{\mcl H}^\gamma(X \cap \mcl T_h^\alpha) - \zeta$. 
\end{lem}

We emphasize that the conclusion of Lemma~\ref{lem-thick-subset} holds a.s.\ for \emph{every} bounded Borel set $X\subset\BB C$, i.e., we do not need to fix $X$ (although $\ol \ep$ is allowed to depend on $X$). A similar comment applies to Lemma~\ref{lem-circle-avg-bound} below.

\begin{proof}[Proof of Lemma~\ref{lem-thick-subset}]
By definition,
\eqb \label{eqn-thick-subset-union}
X \cap \mcl T_h^\alpha \subset \bigcup_{n=1}^\infty \left\{z\in X : \frac{h_\ep(z)}{\log(\ep^{-1})} \in [\alpha-\zeta , \alpha+\zeta] ,\: \forall \ep \in (0,1/n] \right\}. 
\eqe
By the countable stability of Hausdorff dimension, for each large enough $n\in\BB N$, the set inside the union on the right side of~\eqref{eqn-thick-subset-union} has $\dim_{\mcl H}^0$-dimension at least $\dim_{\mcl H}^0 (X \cap \mcl T_h^\alpha)- \zeta$, and the same is true with $\dim_{\mcl H}^\gamma$ in place of $\dim_{\mcl H}^0$.
\end{proof} 

Next, we would like to convert the bounds of~\eqref{lower_b2} to bounds on the circle average process centered at \emph{deterministic} points. The following  lemma asserts that~\eqref{lower_b2} yields bounds on the circle averages at centers of Euclidean or quantum balls of sufficiently small radii which intersect $X^{\alpha, \zeta}$.

\begin{lem} \label{lem-circle-avg-bound}
Let $\alpha \in \BB R$ and $\zeta \in (0,1)$.
Almost surely, for each bounded Borel set $X\subset \BB C$, there exists a $\delta > 0$ (depending on $\alpha,\zeta$, and $X$) such that for each $r \in (0,\delta)$ and each $z \in \BB C$ which lies at  (Euclidean)  distance at most $r$ from a point in $X^{\alpha,\zeta}$, 
\[
\frac{h_{r}(z)}{\log (r^{-1})} \in [m_{\alpha,\zeta}, M_{\alpha,\zeta}] ,
\]
where
\eqb
m_{\alpha,\zeta} := (1-\zeta^2) (\alpha - \zeta) - 3\sqrt{10} \zeta \qquad \text{and} \qquad M_{\alpha,\zeta} := (1-\zeta^2) (\alpha + \zeta) + 3\sqrt{10} \zeta .
\label{mM}
\eqe 
\end{lem}

The proof of Lemma~\ref{lem-circle-avg-bound} uses the following continuity estimate for the circle average process, which is an immediate consequence of~\cite[Lemma 3.15]{ghm-kpz}. 

\begin{lem}[\!\!{\cite{ghm-kpz}}] \label{lem-circle-avg-continuity}
Almost surely, for each bounded Borel set $X\subset \BB C$ and each $\rho \in (0,1/2)$, there exists $\delta = \delta(\rho,X) >0$ such that for each $r \in (0,\delta)$ and each $z,w \in X$ with $|z-w| \leq 2 r$, 
\[ 
|h_r(w) - h_{r^{1 - \rho}}(z)| \leq 3 \sqrt{10 \rho} \log (r^{-1}) 
\] 
\end{lem}

\begin{proof}[Proof of Lemma~\ref{lem-circle-avg-bound}]
By Lemma~\ref{lem-thick-subset}, it is a.s.\ the case that for each bounded Borel set $X\subset \BB C$, there exists $\delta_1 = \delta_1(\alpha,\zeta,X) > 0$ such that
\eqb
\frac{h_r(z)}{ \log (r^{-1})} \in  [\alpha - \zeta, \alpha + \zeta] \qquad \forall z \in X^{\alpha,\zeta}, \quad \forall r \in (0,\delta_1) .
\label{eqn-circle-avg-control-1} 
\eqe 
Also, by Lemma~\ref{lem-circle-avg-continuity} applied with $B_1(X)$ in place of $X$ and with $\rho = \zeta^2$, there exists $\delta_2 = \delta_2(\alpha,\zeta, X) \in (0,1)$ such that
\eqb
|h_r(z) - h_{r^{1- \zeta^2}}(w) | \leq 3 \sqrt{10} \zeta \log (r^{-1})  \qquad \text{$\forall z,w \in B_1(X)$ with $|z-w| \leq 2r$} \quad \forall r \in (0,\delta_2) .
\label{eqn-circle-avg-control-2}
\eqe 
If we choose $\delta < \delta_1^{1/(1-\zeta^2)} \wedge \delta_2$, then combining~\eqref{eqn-circle-avg-control-1} and~\eqref{eqn-circle-avg-control-2} yields,  for all $r \in (0,\delta)$ and each $z \in \BB C$ at (Euclidean) distance at most $r$ from a point $w \in X^{\alpha,\zeta}$, 
\[
h_r(z) \leq h_{r^{1-\zeta^2}}(w)  + 3 \sqrt{10} \zeta \log (r^{-1})  \leq (1-\zeta^2) (\alpha + \zeta) \log (r^{-1}) + 3\sqrt{10} \zeta \log (r^{-1})
\]
and
\[
h_r(z) \geq h_{r^{1-\zeta^2}}(w)  - 3 \sqrt{10} \zeta \log (r^{-1})  \geq (1-\zeta^2) (\alpha - \zeta) \log (r^{-1} )- 3\sqrt{10} \zeta \log( r^{-1}),
\]
as desired.
\end{proof}
 
\subsection{Upper bound for quantum dimension}
\label{sec-ind-upper-thick}

In this subsection we prove the upper bound for $\dim_{\mcl H}^\gamma(X \cap T_h^{\alpha})$ in Theorem~\ref{thm-kpz-independent-thick}.
We first make some reductions.
Since $X$ is independent from $h$, by conditioning on $X$ we can assume without loss of generality that $X$ is deterministic.
By the countable stability of Hausdorff dimension, we can also assume without loss of generality that $X$ is contained in a compact subset of $U$. 
By the Markov property of the whole-plane GFF (see, e.g.,~\cite[Proposition 2.8]{ig4}), if $h$ is a whole-plane GFF normalized so that $h_1(0) = 0$, then $h$ can be coupled with the zero-boundary GFF on $U$ in such a way that the two distributions a.s.\ differ by a random harmonic (hence continuous) function on $U$. 
By Weyl scaling (Axiom~\ref{item-metric-f}), we may therefore assume without loss of generality that $h$ is a whole-plane GFF normalized so that $h_1(0) = 0$.  
We make all of these assumptions throughout this section. 

The proof of the upper bound for $\dim_{\mcl H}^\gamma(X \cap T_h^{\alpha})$ in Theorem~\ref{thm-kpz-independent-thick} uses the following lemma, which is an immediate consequence of~\cite[Proposition 3.9]{lqg-metric-estimates}. 

\begin{lem}
\label{lem-moment0}
For each $p \in (0,4d_\gamma/\gamma^2)$ and each $z\in\BB C$, we have
\eqb \label{eqn-moment0}
\BB E\left[ \left(  e^{-\xi h_{2r}(z)} \sup_{u,v \in B_r(z)} D_h\left( u, v ; B_{2r}(z)   \right)   \right)^p \right] \leq O_r(r^{\xi Q p}) \quad \text{as $r\rta 0$}
\eqe
where the rate of convergence of the $O_r(r^{\xi Q p})$ as $r\rta 0$ depends only on $p$.
\end{lem}
\begin{proof}
The estimate for $z=0$ is~\cite[Proposition 3.9]{lqg-metric-estimates} with $\frk c_r = r^{\xi Q}$, $U = B_2(0)$, and $K = \ol{B_1(0)}$. The estimate for general $z \in \BB C$ then follows from the translation invariance of the law of $h$, modulo additive constant (which implies that $h(\cdot + z)  - h_{2r}(z) \eqD h - h_{2r}(0)$) and Axiom~\ref{item-metric-coord} (applied with $\phi(w) = w-z$).
\end{proof}

\begin{proof}[Proof of Theorem~\ref{thm-kpz-independent-thick}, upper bound]
Let $\zeta >0$ and define the set $X^{\alpha, \zeta}$ as in~\eqref{lower_b2}.
Also let $m_{\alpha,\zeta}$ and $M_{\alpha,\zeta}$ be defined as in~\eqref{mM}, so that $m_{\alpha,\zeta}$ and $M_{\alpha,\zeta}$ each converge to $\alpha$ as $\zeta\rta 0$. 
Let $\ell \in [0,d_\gamma]$ satisfy
\eqb
\ell > 
\begin{dcases} 
\frac{\dim_{\mcl H}^0 X - (m_{\alpha,\zeta}^2 \wedge M_{\alpha,\zeta}^2)/2}{\xi (Q - M_{\alpha,\zeta})}
,\quad &\text{if} \: \alpha \neq 0   \\
\frac{\dim_{\mcl H}^0 X }{\xi (Q - M_{0,\zeta})},\quad  &\text{if} \: \alpha = 0
\end{dcases}
\label{eqn-ell-cond}
\eqe
We claim that a.s.\  
\eqb
  \dim_{\mcl H}^\gamma X^{\alpha,\zeta} \leq \ell .
\label{eqn-ell-result}
\eqe
Since $\dim_{\mcl H}^\gamma (X\cap \mcl T_h^\alpha) \leq \dim_{\mcl H}^\gamma X^{\alpha,\zeta}  + \zeta$, once~\eqref{eqn-ell-result} is established we can let $\ell$ decrease to the right side of~\eqref{eqn-ell-cond} and then $\zeta \rta 0$ to get that a.s.\ $\dim_{\mcl H}^\gamma (X\cap \mcl T_h^\alpha) \leq \frac{1}{\xi(Q-\alpha)}   \max\left\{ \dim_{\mcl H}^0 X - \alpha^2/2, 0\right\} $, as required. 

We will now prove~\eqref{eqn-ell-result}. Since $\ell$ satisfies~\eqref{eqn-ell-cond}, we can choose $x \in (\dim_{\mcl H}^0 X, \xi (Q - M_{\alpha,\zeta}) \ell + (m_{\alpha,\zeta}^2 \wedge M_{\alpha,\zeta}^2)/2)$ when $\alpha \neq 0$, or $x \in (\dim_{\mcl H}^0 X, \xi (Q - M_{0,\zeta}) \ell)$ when $\alpha = 0$. By~\eqref{eqn-thick-dim} and the definition of Euclidean Hausdorff dimension, for such a choice of $x$, the following is true: for each $\delta > 0$, we can choose a covering of $X$ by Euclidean balls $\{B_{r_i}(z_i)\}_{i \in \BB N}$ (depending on $\delta$) 
which satisfies
\eqb
\sum_{i \in \BB N} r_i^x < \delta .
\label{eqn-euclidean-cond} 
\eqe
Since $X$ is assumed to be deterministic, we can choose $\{B_{r_i}(z_i)\}_{i\in\BB N}$ in a deterministic manner.
Let
\eqbn
I^{\alpha, \zeta} := \{ i \in \BB N :  B_{ r_i}(z_i) \cap X^{\alpha,\zeta} \neq \emptyset \}.
\eqen
Then $\{B_{r_i}(z_i)\}_{i \in I^{\alpha,\zeta}}$ is a covering of $X^{\alpha,\zeta}$. 
We will show that, for each $\ep > 0$ and any covering $\{B_{r_i}(z_i)\}_{i \in \BB N}$ of $X$ satisfying~\eqref{eqn-euclidean-cond} with $\delta=\ep^2$,
\eqb
\BB P \left[\sum_{i \in I^{\alpha,\zeta}} \left( \sup_{u,v \in B_{r_i}(z_i)} D_h(u,v) \right)^\ell > \ep \right] < 2\ep .
\label{eqn-quantum-tail}
\eqe
We will then deduce~\eqref{eqn-ell-result} by applying the Borel-Cantelli lemma (recall that the balls $B_{r_i}(z_i)$ depend on $\ep$). 

In order to bound the sum in~\eqref{eqn-quantum-tail}, we will prove an upper bound for the expectation of each term in the sum truncated on a global regularity event of probability at least $1 - \ep$. This upper bound will yield~\eqref{eqn-quantum-tail} by Markov's inequality. We now outline how we will derive this upper bound.  
\begin{itemize}
\item We write each term of the sum in~\eqref{eqn-quantum-tail} as
\[
\left( e^{-\xi h_{2r_i}(z_i)} \sup_{u,v \in B_{r_i}(z_i)} D_h(u,v; B_{2r_i}(z_i)) \right)^\ell  \times  e^{\xi \ell h_{2r_i}(z_i)}.
\]
\item
We use Lemma~\ref{lem-moment0} to bound the first factor in this product.  
\item
To handle the $e^{\xi \ell h_{2r_i}(z_i)}$ term, we will use Lemma~\ref{lem-circle-avg-bound} to show that, since $B_{r_i}(z_i) \cap X^{\alpha,\zeta} \neq \emptyset$, the circle average $e^{\xi \ell h_{2r_i}(z_i)}$ is bounded from above by $(2 r_i)^{-\xi M_{\alpha,\zeta}}$. 
\item
Finally, when $\alpha \neq 0$, we use the fact that for each fixed $i\in\BB N$, the ball $B_{r_i}(z_i)$ is unlikely to intersect $ X^{\alpha,\zeta}$ at all, since --- again, by Lemma~\ref{lem-circle-avg-bound} --- if $B_{r_i}(z_i)$ intersects $X^{\alpha,\zeta}$ then the circle average $e^{\xi \ell h_{2r_i}(z_i)}$ must deviate from its typical value.
The probability that this is the case can be bounded above using the Gaussian tail bound. 
\end{itemize}

Having outlined our strategy for proving~\eqref{eqn-quantum-tail}, we now proceed with the proof itself. 
By~\eqref{eqn-euclidean-cond}, we have $r_i < \delta^{1/x}$ for each $i \in \BB N$.  Thus, by Lemma~\ref{lem-circle-avg-bound}, if we choose $\delta$ sufficiently small, then it holds with probability at least $1 - \ep$ that
\[
\frac{h_{2 r_i}(z)}{ \log( (2 r_i)^{-1})} \in [ m_{\alpha,\zeta} , M_{\alpha,\zeta} ]  \qquad \forall i \in I^{\alpha,\zeta} .
\]
By Markov's inequality,
\allb \nonumber
&\BB P\left[ \sum_{i \in I^{\alpha,\zeta}} \left( \sup_{u,v \in B_{r_i}(z_i)} D_h(u,v; B_{2r_i}(z_i)) \right)^\ell  >\ep \right]\\
&\qquad <
\ep + \ep^{-1}   \sum_{i \in \BB N} \BB E \left[ \left( \sup_{u,v \in B_{r_i}(z_i)} D_h(u,v; B_{2r_i}(z_i)) \right)^\ell \BB{1}_{\frac{h_{2r_i}(z)}{ \log ((2r_i)^{-1})} \in [ m_{\alpha,\zeta} , M_{\alpha,\zeta} ] } \right]. \label{upper-tail}
\alle
To bound the summand 
\eqb
\BB E \left[ \left( \sup_{u,v \in B_{r_i}(z_i)} D_h(u,v; B_{2r_i}(z_i)) \right)^\ell \BB{1}_{\frac{h_{2r_i}(z)}{ \log (2r_i)^{-1}} \in [ m_{\alpha,\zeta}, M_{\alpha,\zeta} ] } \right]
\label{eqn-prod-exp}
\eqe
in the latter expression, we first decompose the summand into a product of expectations.  The random variable $h_{2 r_i}(z_i)$ is independent from $(h - h_{2 r_i}(z))|_{B_{2 r_i}(z_i)}$~\cite[Section 3.1]{shef-kpz}. By Axioms~\ref{item-metric-local} and~\ref{item-metric-f},  the internal metric
\eqbn
e^{-\xi h_{2r_i}(z_i)} D_h\left( u,v ; B_{2r_i}(z_i) \right) 
= D_{h-h_{2r_i}(z_i)} \left( u,v ; B_{2r_i}(z_i) \right) 
\eqen
is a.s.\ given by a measurable function of $(h-h_{2r_i}(z_i))|_{B_{2r_i}(z_i)}$, so is also independent from $ h_{2r_i}(z_i)$. Thus,~\eqref{eqn-prod-exp} equals
\[
\BB E \left[  \left( e^{-\xi h_{2r_i}(z_i)} \sup_{u,v \in B_{r_i}(z_i)} D_h(u,v; B_{2r_i}(z_i)) \right)^\ell \right] \BB E \left[ e^{\ell \xi h_{2r_i}(z_i)} \BB{1}_{\frac{h_{2r_i}(z)}{ \log (2 r_i)^{-1}} \in [ m_{\alpha,\zeta} , M_{\alpha,\zeta} ] }  \right] 
\]
We now treat each of the two expectations separately.  On the one hand, Lemma~\ref{lem-moment0} gives
\eqb
\BB E \left[  \left( e^{-\xi h_{2r_i}(z_i)} \sup_{u,v \in B_{r_i}(z_i)} D_h(u,v; B_{2r_i}(z_i)) \right)^\ell \right] \leq r_i^{\xi Q \ell + o_\delta(1)} ,
\label{first-expectation}
\eqe
where here we recall that $r_i \leq \delta^{1/x}$, so $o_{r_i}(1)\leq o_\delta(1)$. 
On the other hand,
\eqb
\BB E \left[ e^{\ell \xi h_{2r_i}(z_i)} \BB{1}_{\frac{h_{2r_i}(z)}{ \log (2 r_i)^{-1}} \in  [ m_{\alpha,\zeta} , M_{\alpha,\zeta} ] }  \right]   
\leq (2r_i)^{- \ell \xi  M_{\alpha,\zeta} } \BB P\left[\frac{h_{2r_i}(z)}{ \log (2 r_i)^{-1}} \in [ m_{\alpha,\zeta} , M_{\alpha,\zeta} ] \right];
\label{second-expectation}
\eqe
and, by the Gaussian tail bound,
\eqb
\BB P\left[\frac{h_{2r_i}(z)}{ \log (2 r_i)^{-1}} \in [ m_{\alpha,\zeta} , M_{\alpha,\zeta} ] \right] \leq  (2 r_i)^{ (m_{\alpha,\zeta}^2 \wedge M_{\alpha,\zeta}^2)/2} \qquad \text{when $\alpha \neq 0$} .
\label{eqn-gaussian-tail-bound}
\eqe
Combining~\eqref{first-expectation},~\eqref{second-expectation} and~\eqref{eqn-gaussian-tail-bound}, we deduce that~\eqref{eqn-prod-exp} is at most 
\[
\begin{dcases} 
r_i^{\ell \xi (Q - M_{\alpha,\zeta})  + (m_{\alpha,\zeta}^2 \wedge M_{\alpha,\zeta}^2)/2 + o_{\delta}(1)}
,\quad &\text{if} \: \alpha \neq 0   \\
r_i^{ \ell \xi (Q - M_{0,\zeta}) + o_{\delta}(1)},\quad  &\text{if} \: \alpha = 0
\end{dcases}
\]
which (in either case) is strictly less than $r_i^x$ by the choice of $x$.  Thus,~\eqref{upper-tail} is at most 
\[
\ep + \ep^{-1} \sum_{i \in \BB N} r_i^x
\]
for $\delta$ sufficiently small. By~\eqref{eqn-euclidean-cond}, this last quantity is at most $2\ep$ provided $\delta \leq \ep^2$. This proves~\eqref{eqn-quantum-tail}.

We deduce~\eqref{eqn-ell-result} by setting $\ep = 2^{-k}$ and applying the Borel-Cantelli lemma.
\end{proof}

\subsection{Lower bound for quantum dimension}
\label{sec-ind-lower-thick}

In this subsection we prove the lower bound  for  $\dim_{\mcl H}^\gamma X \cap \mcl T_h^{\alpha}$ in Theorem~\ref{thm-kpz-independent-thick}.
In fact, we will prove the following stronger statement which does not require that $X$ is independent from $h$.

\begin{prop} \label{prop-thick-lower}
Assume that we are in the setting of Theorem~\ref{thm-kpz-independent-thick}.
Almost surely, for every Borel set $X\subset U$,
\eqb \label{eqn-thick-lower}
\dim_{\mcl H}^\gamma(X \cap \mcl T_h^{\alpha}) \geq \frac{\dim_{\mcl H}^0 (X \cap \mcl T_h^{\alpha}) }{\xi(Q-\alpha)} .
\eqe
\end{prop}

The lower bound for $\dim_{\mcl H}^\gamma(X \cap \mcl T_h^{\alpha})$ in Theorem~\ref{thm-kpz-independent-thick} is an immediate consequence of Proposition~\ref{prop-thick-lower} and~\eqref{eqn-thick-dim}.

As in Section~\ref{sec-ind-upper-thick}, it suffices to prove Proposition~\ref{prop-thick-lower} only in the special case when $h$ is a whole-plane GFF and $X$ is contained in some deterministic bounded open subset $U$ of $\BB C$. The key ingredient for the proof of Proposition~\ref{prop-thick-lower} is a uniform lower bound for the $D_h$-distance between two concentric Euclidean balls in terms of the circle average. The fact that this holds \emph{uniformly} over all possible Euclidean balls is the reason why we do not need to require that $X$ is independent from $h$ in Proposition~\ref{prop-thick-lower}.

\begin{lem} \label{lem-annulus-dist-uniform}
Let $U\subset \BB C$ be a bounded open set and let $\zeta >0$.
It holds with probability tending to 1 as $\ep \rta 0$ that  
\eqb
D_h\left( \bdy B_r(z) , \bdy B_{r/2}(z) \right) \geq r^{\xi Q +\zeta} e^{\xi h_r(z)} ,\quad\forall z\in U, \quad\forall r \in (0,\ep] .
\eqe
\end{lem}

To prove Lemma~\ref{lem-annulus-dist-uniform}, we will use the following lower bound for the $D_h$-distance between two concentric Euclidean balls that follows directly from~\cite[Proposition 3.1]{lqg-metric-estimates}, together with a union bound argument.  

\begin{lem}
\label{lem-prop3.1}
For each $\zeta > 0$ and each $z \in \BB{C}$, it holds with superpolynomially high probability as $r \rta 0$ (i.e., with probability $1 - O_r(r^{-p})$ for every $p>0$), at a rate which depends on $\zeta$ but not on $z$, that
\eqb \label{eqn-prop3.1}
D_h\left( \bdy B_{7r/8}(z),  \bdy B_{5r/8}(z) \right) \geq r^{  \xi Q + \zeta} e^{\xi h_r(z)} .
\eqe 
\end{lem}
\begin{proof}
The estimate for $z=0$ is an application of~\cite[Proposition 3.1]{lqg-metric-estimates} with $\frk c_r = r^{\xi Q}$, $A = r^{-\zeta}$, $K_1 = \bdy B_{7/8}(0)$, and $K_2 = \bdy B_{5/8}(0)$. 
The estimate for general $z \in \BB C$ then follows from the translation invariance of the law of $h$ (modulo additive constant) and Axiom~\ref{item-metric-coord}.
\end{proof}

\begin{proof}[Proof of Lemma~\ref{lem-annulus-dist-uniform}]
By Lemma~\ref{lem-prop3.1} and a union bound, it holds with superpolynomially high probability as $r \rta 0$ that 
\eqb
D_h\left( \bdy B_{7r/8}(z) , \bdy B_{5r/8}(z) \right) \geq r^{\xi Q +\zeta/2} e^{\xi h_r(z)} , \quad \forall z\in B_1(U) \cap (r^2 \BB Z^2) .
\eqe
Setting $r = 1/n$ and taking a union bound over all $n\in\BB N$ with $n\geq (2\ep)^{-1}$ shows that with superpolynomially high probability as $\ep\rta 0$,
\eqb \label{eqn-annulus-union}
D_h\left( \bdy B_{7 /(8n)}(z) , \bdy B_{5 /(8 n) }(z) \right) \geq n^{-(\xi Q +\zeta/2) } e^{\xi h_{1/n}(z)} , \quad \forall z\in B_1(U) \cap \left(\frac{1}{n^2} \BB Z^2 \right), \quad \forall n \geq (2\ep)^{-1} .
\eqe
By a standard circle average continuity estimate (see, e.g.,~\cite[Lemma 3.13]{ghm-kpz}, which itself follows from~\cite[Proposition 2.1]{hmp-thick-pts}), it holds with probability tending to 1 as $\ep\rta 0$ that 
\eqb \label{eqn-circle-avg-cont}
|h_r(z) - h_{r'}(z')| \leq \frac{|(z,r) - (z',r')|^{99/200}}{r^{1/2}} , 
\quad \forall z,z'\in U ,
\quad \text{$\forall r,r' \in (0,\ep]$ with $r/r' \in [1/2,2]$}.
\eqe

Henceforth assume that~\eqref{eqn-annulus-union} and~\eqref{eqn-circle-avg-cont} both hold, which happens with probability tending to 1 as $\ep\rta 0$.
If $\ep$ is sufficiently small, then for each $r\in (0,\ep]$ and each $z\in U$ we can find $n\in\BB N$ such that $|1/n  - r | \leq 4r^2$ and $z' \in B_1(U) \cap \left(\frac{1}{n^2} \BB Z^2 \right)$ such that $|z-z'| \leq 4r^2$. 
By~\eqref{eqn-circle-avg-cont} with $r' = 1/n$, 
\eqb  \label{eqn-circle-avg-compare}
|h_r(z) - h_{1/n}(z')| \leq \frac{(8r^2)^{99/200}}{r^{1/2}}  \leq C r^{49/100} ,
\eqe 
where $C$ is a universal constant. By our choice of $z'$, if $\ep$ is sufficiently small then $   B_{7 /(8n)}(z') \setminus  B_{5 /(8 n) }(z') \subset B_r(z) \setminus B_{r/2}(z)$. By~\eqref{eqn-annulus-union} with $z'$ in place of $z$, followed by~\eqref{eqn-circle-avg-compare}, we obtain 
\alb
D_h\left( \bdy B_r(z) , \bdy B_{r/2}(z) \right)
\geq D_h\left( \bdy B_{7 /(8n)}(z') , \bdy B_{5 /(8 n) }(z') \right)
&\geq n^{-(\xi Q +\zeta/2) } e^{\xi h_{1/n}(z')} \notag \\
&\geq 2 r^{\xi Q + \zeta/2} e^{-C \xi r^{49/100}} e^{\xi h_r(z)} , 
\ale
which is at least $r^{\xi Q + \zeta} e^{\xi h_r(z)}$ provided $\ep$ (and hence also $r$) is made to be sufficiently small. 
\end{proof}

\begin{proof}[Proof of Theorem~\ref{thm-kpz-independent-thick}, lower bound]
As noted above, we can assume without loss of generality that $h$ is a whole-plane GFF and we can restrict attention to Borel sets $X$ which are contained in some deterministic bounded open subset $U$ of $\BB C$. 
Let $\zeta > 0$ be a small parameter which we will eventually send to zero, and for each Borel set $X \subset U$ define $X^{\alpha,\zeta}$ as in Lemma~\ref{lem-thick-subset}.
  
We claim that with $m_{\alpha,\zeta}$ as in~\eqref{mM}, it is a.s.\ the case that for every Borel set $X\subset U$ and every 
$\ell \in \left( \dim_{\mcl H}^\gamma X^{\alpha,\zeta} , d_\gamma \right]$,
\eqb
(\xi (Q - m_{\alpha,\zeta}) +  \zeta) \ell  \geq \dim_{\mcl H}^0 X^{\alpha,\zeta} .
\label{eqn-ell-result2}
\eqe
Once~\eqref{eqn-ell-result2} is established, we send $\ell \rta \dim_{\mcl H}^\gamma X^{\alpha,\zeta} $ and then send $\zeta \rta 0$ and recall that $\dim_{\mcl H}^0 X^{\alpha,\zeta} \geq \dim_{\mcl H}^0\left(X\cap \mcl T_h^\alpha\right) - \zeta$ to obtain~\eqref{eqn-thick-lower}.

We will now prove~\eqref{eqn-ell-result2}.  
For $\wt r > 0$ and $z\in\BB C$, write $\wt B_{\wt r}(z)$ for the $D_h$-metric ball of radius $\wt r$ centered at $z$. 
By the definition of quantum dimension, for any Borel set $X\subset U$ and any $\ell \in \left( \dim_{\mcl H}^\gamma X^{\alpha,\zeta} , d_\gamma \right]$ the following is true: for each $\delta > 0$, we can choose a covering of $X^{\alpha,\zeta}$ by $D_h$-metric balls $\{\wt B_{\wt r_i}(z_i)\}_{i \in \BB N}$ whose $D_h$-radii satisfy
\eqb
\sum_{i \in \BB N} \wt r_i^\ell < \delta .
\label{eqn-quantum-cond}
\eqe 
For each $i\in\BB N$, define the Euclidean radius
\eqb \label{eqn-euclidean-rad}
r_i:= \sup_{u \in   \wt B_{\wt r_i}(z_i)} |u-z_i| .
\eqe
We will show that it is a.s.\ the case that, for each $\ep > 0$ and each Borel set $X\subset U$, there exists $\delta = \delta(\ep , X) > 0$ such that for any covering $\{\wt B_{\wt r_i}(z_i)\}_{i \in \BB N}$ of $X^{\alpha,\zeta}$ by quantum balls satisfying~\eqref{eqn-quantum-cond},
\eqb
 \sum_{i \in \BB N} r_i^{(\xi (Q - m_{\alpha,\zeta}) - \zeta) \ell} < \ep . 
\label{eqn-euclidean-tail}
\eqe 
We will then send $\ep \rta 0$ to obtain~\eqref{eqn-ell-result2}. 

Roughly speaking, our strategy for proving~\eqref{eqn-euclidean-tail} is as follows.  By the definition~\eqref{eqn-euclidean-rad} of $r_i$,
\eqb
D_h\left( \bdy B_{r_i}(z_i), \bdy B_{r_i/2}(z_i) \right) \leq \wt r_i \qquad \forall i \in \BB N .
\label{eqn-by-def}
\eqe
Moreover, Lemma~\ref{lem-annulus-dist-uniform} allows us to bound from below the quantum distances $D_h( B_{r_i}(z), B_{r_i/2}(z))$ for $i \in \BB N$ in terms of $r_i$ and $h_{r_i}(z_i)$.  Finally, to convert this bound to an upper bound on $r_i$ in terms of $\wt r_i$, we use Lemma~\ref{lem-circle-avg-bound} to lower-bound $h_{r_i}(z_i)$.  The application of Lemma~\ref{lem-circle-avg-bound} is slightly more subtle here than in the proof of the upper bound in the previous subsection, since here we know a priori that the centers $z_i$ for $i \in  \BB N$ are close to points in $X^{\alpha,\zeta}$ in the $D_h$ metric, not in the Euclidean metric. This is not a problem, however, since the Euclidean metric is a.s. locally bi-H\"older with respect to $D_h$~\cite[Theorem 1.7]{lqg-metric-estimates}, so pairs of points which are close w.r.t.\ $D_h$ are also close w.r.t.\ the Euclidean metric, in a sense which is sufficiently strong for our purposes. 

Having described our strategy for proving~\eqref{eqn-euclidean-tail}, we now proceed with the proof itself. 
By~\eqref{eqn-quantum-cond}, we have $\wt r_i <\delta^{1/\ell}$ for each $i \in \BB N$.  As noted in Remark~\ref{remark-holder},~\cite[Theorem 1.7]{lqg-metric-estimates} asserts that the identity map from $(U,D_h)$ to $U$, equipped with the Euclidean metric is a.s. locally H\"older continuous with any exponent $\chi <\xi^{-1}(Q+2)^{-1}$.  Thus, if we fix such a $\chi$, then a.s.\ for each  $\delta > 0$ smaller than some threshold independent of which set $X$ we consider,
\[
r_i < \wt r_i^\chi < \delta^{\chi/\ell} \qquad \forall i \in  \BB N .
\]
By Lemma~\ref{lem-circle-avg-bound}, this implies that a.s.\ for each $\delta > 0$ smaller than some threshold (that may depend on $X$), 
\eqb
\frac{h_{r_i}(z)}{ \log r_i^{-1}} \geq m_{\alpha,\zeta}   \qquad \forall i \in \BB N.
\label{eqn-circle-avg-control-I-alpha-zeta}
\eqe
By Lemma~\ref{lem-annulus-dist-uniform}, it is a.s.\ the case that for each small enough $\delta > 0$,
\[
  D_h\left( \bdy B_{r_i}(z), \bdy B_{r_i/2}(z) \right)   \geq r_i^{\xi Q +\zeta} e^{\xi h_{r_i}(z_i)}   \qquad \forall i \in \BB N .
\]
Applying~\eqref{eqn-by-def}, we deduce that it is a.s.\ the case that for each small enough $\delta >0$, 
\[
  \wt r_i   \geq r_i^{\xi Q +\zeta} e^{\xi h_{r_i}(z_i)}   \qquad \forall i \in \BB N .
\]
By~\eqref{eqn-circle-avg-control-I-alpha-zeta}, this implies that a.s.\ for each small enough $\delta > 0$,
\[
r_i^{\xi Q + \zeta - \xi m_{\alpha,\zeta}} \leq    r_i^{\xi Q +\zeta} e^{\xi h_{r_i}(z_i)} \leq \wt r_i \qquad \forall i \in  \BB N 
\]
and therefore
\[
\sum_{i \in \BB N} r_i^{(\xi (Q - m_{\alpha,\zeta}) + \zeta) \ell} \leq \sum_{i \in \BB N} \wt r_i^\ell < \delta < \ep.
\]
This proves~\eqref{eqn-euclidean-tail}. Sending $\ep\rta 0$ now gives~\eqref{eqn-ell-result2}, which, as explained above, concludes the proof.
\end{proof}

\section{Proof of the KPZ formula for sets independent from $h$}
\label{sec-kpz-independent}

In this section we will prove Theorem~\ref{thm-kpz-independent}.
As in Section~\ref{sec-ind-upper-thick}, we can assume without loss of generality that $X$ is bounded and deterministic and that $h$ is a whole-plane GFF normalized so that $h_1(0) = 0$.
The proof of the upper bound for $\dim_{\mcl H}^0 X$ is an immediate consequence of  the lower bound on $\dim_{\mcl H}^\gamma X \cap \mcl T_h^{\alpha}$ in Theorem~\ref{thm-kpz-independent-thick} proved in Section~\ref{sec-ind-lower-thick}:

\begin{proof}[Proof of Theorem~\ref{thm-kpz-independent}, upper bound]
The lower bound on $\dim_{\mcl H}^\gamma\left( X \cap \mcl T_h^{\alpha} \right)$ in Theorem~\ref{thm-kpz-independent-thick}  is maximized when $\alpha = Q - \sqrt{Q^2 - 2 \dim_{\mcl H}^0 X}$, in which case we have
\eqb
\dim_{\mcl H}^{\gamma} \left( X \cap \mcl T_h^\alpha \right) \geq \frac{1}{\xi} \left( Q - \sqrt{Q^2 - 2 \dim_{\mcl H}^0 X} \right).
\label{laststep}
\eqe
Since $\dim_{\mcl H}^{\gamma} \left( X \cap \mcl T_h^\alpha \right) \leq \dim_{\mcl H}^{\gamma} X$, this gives the desired upper bound for $\dim_{\mcl H}^0 X$. 
\end{proof}

To prove the lower bound for $\dim_{\mcl H}^0 X$, we begin by eliminating the $ e^{-\xi h_{2r}(z)}$ term in the bound of  Lemma~\ref{lem-moment0} to obtain a bound on moments of the quantum diameter of a Euclidean ball.

\begin{lem} \label{lem-diam-moment}
Let $p \in \left( 0 , 4d_\gamma/\gamma^2 \right)$. For each compact set $K\subset \BB C$, each $z\in K$,  
\eqb \label{eqn-diam-moment}
\BB E\left[ \left( \sup_{u,v\in B_r(z)} D_h(u,v) \right)^p \right] \leq r^{\xi Q p - \xi^2 p^2 /2 + o_r(1)} \quad \text{as $r\rta 0$},
\eqe
with the rate of the $o_r(1)$ depending only on $p$ and $K$. 
\end{lem}
\begin{proof}
For $z\in\BB C$, the random variable $h_1(z) - h_{2r}(z)$ is centered Gaussian with variance $\log\ep^{-1}   - \log 2$ and is independent from $(h-h_{2r}(z))|_{B_{2r}(z)}$~\cite[Section 3.1]{shef-kpz}. 
By Axioms~\ref{item-metric-local} and~\ref{item-metric-f}, the internal metric
\eqbn
e^{-\xi h_{2r}(z)} D_h\left( u,v ; B_{2r}(z) \right) 
= D_{h-h_{2r}(z)} \left( u,v ; B_{2r}(z) \right) 
\eqen
is a.s.\ determined by $(h-h_{2r}(z))|_{B_{2r}(z)}$, so is also independent from $h_1(z) - h_{2r}(z)$. 

By Lemma~\ref{lem-moment0} together with the formula $\BB E[e^Z] = e^{\op{Var}(Z)/2}$ for a Gaussian random variable $Z$, 
\allb \label{eqn-diam-moment-scaled}
& \BB E\left[ \left( e^{-\xi h_1(z)} \sup_{u,v\in B_r(z)}  D_h\left( u, v ; B_{2r}(0)   \right) \right)^p \right]  \notag\\
&\qquad= \BB E\left[ e^{\xi p (h_{2r}(z) - h_1(z))} \right] \BB E\left[ \left( e^{-\xi h_{2r}(z)} \sup_{u,v\in B_r(z)} D_h\left( u, v ; B_{2r}(0)   \right) \right)^p \right] \notag \\
&\qquad= O_r\left( r^{\xi Q p - \xi^2 p^2 / 2} \right) .
\alle  

Since $D_h(u,v ;B_{2r}(0)) \geq D_h(u,v)$,~\eqref{eqn-diam-moment-scaled} implies~\eqref{eqn-diam-moment} but with an extra factor $e^{-\xi h_1(z)}$ inside the expectation.
To get rid of this factor, we observe that $ h_1(z) $ is centered Gaussian with variance bounded above by a constant depending only on $K$. 
By H\"older's inequality and~\eqref{eqn-diam-moment-scaled}, for any $q  >  1$ such that $q p  < 4d_\gamma/\gamma^2 $, 
\allb \label{eqn-diam-moment-holder}
& \BB E\left[ \left(  \sup_{u,v\in B_r(z)}  D_h\left( u, v ; B_{2r}(0)   \right) \right)^p \right] \notag \\
&\qquad \leq \BB E\left[ e^{\xi q h_1(z)/(q-1)} \right]^{1-1/q} \BB E\left[ \left( e^{- \xi h_1(z)}  \sup_{u,v\in B_r(z)}  D_h\left( u, v ; B_{2r}(0)   \right) \right)^{q p} \right]^{1/q}  \notag \\
&\qquad= O_r\left( r^{\xi Q p   - \xi^2 p^2 q / 2} \right) .
\alle  
Since $q$ can be made arbitrarily close to 1, we can now send $q\rta 1$ at a sufficiently slow rate as $r\rta 0$ to get~\eqref{eqn-diam-moment} (note that $D_h(u,v) \leq D_h(u,v ; B_{2r}(z))$).
\end{proof}

 \begin{proof}[Proof of Theorem~\ref{thm-kpz-independent}, lower bound]
Proving the lower bound on $\dim_{\mcl H}^0 X$ in~\eqref{kpz-independent} is equivalent to showing that, if $\ell \in [0,d_\gamma]$ satisfies $\xi Q \ell - \xi^2 \ell^2/2 > \dim_{\mcl H}^0 X$, then $\ell \geq \dim_{\mcl H}^{\gamma}  X$ almost surely.  

Fix such an $\ell$, and let $x  \in \left( \dim_{\mcl H}^0 X , \xi Q \ell - \xi^2 \ell^2/2\right)$.  
By definition of (Euclidean) Hausdorff dimension, for each $\ep > 0$, we can choose a collection $\{z_i\}_{i \in \BB N}$ of points of $X$ and radii $r_i  > 0$ such that the Euclidean balls $\{B_{r_i}(z_i)\}_{i\in \BB N}$ cover $X$ and the $r_i$ satisfy $\sum_{i \in \BB N} r_i^x < \ep$.  Since $X$ is assumed to be bounded and deterministic, we can apply Lemma~\ref{lem-diam-moment} with $p =\ell$ to get that, if $\ep \geq \max_{i \in \BB N} r_i^x$ is sufficiently small, then
\alb
\BB{E} \left[ \sum_{i \in \BB N} \left( \sup_{u,v \in B_{r_i}(z_i)} D_h(u,v) \right)^\ell \right]  
\leq  \sum_{i \in \BB N} r_i^{  \xi Q \ell - \xi^2 \ell^2/2 + o_\ep(1) }
\leq \sum_{i \in \BB N} r_i^x < \ep.
\ale
By Markov's inequality, this implies that for small enough $\ep$, 
\eqbn 
\BB P\left[ \sum_{i \in \BB N} \left( \sup_{u,v  \in B_{r_i}(z_i)} D_h(u,v ) \right)^{\ell}  \leq \sqrt{\ep} \right] \geq 1 - \sqrt \ep .
\eqen
Setting $\ep = 2^{-n}$ and applying the Borel-Cantelli lemma now shows that a.s.\ $\ell \geq \dim_{\mcl H}^{\gamma}  X$, as desired.
\end{proof}

\section{Proof of the worst-case KPZ formula}
\label{sec-kpz-correlated}

We will now prove Theorem~\ref{thm-kpz-correlated}. 
As in previous subsections, we may assume that $U = \BB{C}$ and $h$ a whole-plane GFF normalized to have $h_1(0) = 0$; and we may restrict attention to fractals which are contained in some deterministic compact set (we can no longer assume that $X$ is deterministic, though, since now $X$ is not required to be independent from $h$).
In fact, due to the translation invariance of the law of $h$, modulo additive constant, and Axioms~\ref{item-metric-f} and~\ref{item-metric-coord}, we can assume without loss of generality that this compact set is $[0,1]^2$.  
This will be convenient since we are going to work with dyadic squares. 
Throughout this section, we will denote the side length of a Euclidean square $S$ by $|S|$.

\subsection{Upper bound for quantum dimension}
\label{sec-wc-quantum}
 
For $n\in\BB N$, we let $\mcl D_n$ be the set of $2^{-n}\times 2^{-n}$ squares contained in $[0,1]^2$ with corners in $2^{-n}\BB Z^2$.
The idea of the proof of~\eqref{upper-bound-quantum} is as follows.
We will consider a Borel set $X \subset [0,1]^2$, a number $x >  \dim_{\mcl H}^0 X$, and a number $\ell$ greater than the right side of~\eqref{upper-bound-quantum}. 
We cover $X$ by a collection $\mcl S$ of small dyadic squares for which $\sum_{S\in\mcl S} |S|^x \leq 1$, and we seek to show that $\sum_{S\in\mcl S} (\sup_{u,v\in S} D_h(u,v))^\ell $ is small. 
To do this, we will consider the ``worst-case scenario" where for each large enough $n$, the squares in $\mcl S\cap \mcl D_n$ are the $\lfloor 2^{x n} \rfloor$ squares of $\mcl D_n$ with the \emph{largest} $D_h$-diameters. 
Lemma~\ref{lem-eucl-square-count} just below will allow us to show that the sum of the $D_h$-diameters of these $\lfloor 2^{x n} \rfloor$ largest squares is bounded above by a positive power of $2^{-n}$. Summing over all $n\in\BB N$ will then give the desired bound for $\dim_{\mcl H}^{\gamma} X$. 
 
\begin{lem} \label{lem-eucl-square-count}
For each $s \in (0,\xi Q)$ and each $\zeta \in (0,1)$, there exists $\alpha =\alpha(s,\zeta) > 0$ such that for each $n\in\BB N$, it holds with probability at least $1-O_n(2^{-\alpha n})$ that
\eqb \label{eqn-eucl-square-count}
\#\left\{ S\in\mcl D_n : \sup_{u,v\in S} D_h(u,v) > 2^{-s n} \right\} \leq 2^{ \left(2 -  \frac{(\xi Q - s)^2}{2\xi^2} + \zeta \right) n } .
\eqe
In particular, if $s < \xi(Q-2)$, then there exists $\alpha'=\alpha'(s) > 0$ such that with probability at least $1- O_n(2^{-\alpha' n})$, the set in~\eqref{eqn-eucl-square-count} is empty. 
\end{lem}
\begin{proof}
By~\cite[Lemma 3.19]{lqg-metric-estimates} (applied with $\BB r = 1$ and $K = [0,1]^2$), for each fixed $S\in \mcl D_n$,
\eqb
\BB P\left[\sup_{u,v\in S} D_h(u,v) > 2^{-s n} \right] \leq 2^{-\left(   \frac{(\xi Q - s)^2}{2\xi^2}   - o_n(1) \right) n} . 
\eqe
The estimate~\eqref{eqn-eucl-square-count} follows from this and Markov's inequality. 
The last statement follows since if $s < \xi(Q-2)$, then the right side of~\eqref{eqn-eucl-square-count} is less than 1 for a small enough choice of $\zeta$ (depending only on $s$ and $\gamma$). 
\end{proof}

\begin{proof}[Proof of~\eqref{upper-bound-quantum}] 
\noindent\textit{Step 1: setup.}
Fix a small parameter $\zeta \in (0,1)$ (which we will eventually send to zero) and a partition of $[\xi(Q-2)-\zeta , \xi Q]$ with partition points
\eqb \label{eqn-partition}
\xi(Q-2)-\zeta = s_0 < \dots < s_K = \xi Q \quad \text{with} \quad \max_{k=1,\dots,K} (s_k  -s_{k-1}) \leq \zeta .
\eqe 
We can arrange that the $s_k$'s depend only on $\zeta$. 
Throughout the proof, $o_\zeta(1)$ denotes a deterministic quantity which depends only on $\zeta$ and which converges to zero as $\zeta \rta 0$. 

We partition the set of dyadic squares $\mcl D_n$ according to their quantum diameters by defining
\eqb \label{eqn-quantum-square-partition}
\mcl D_{n,k} 
:=  \left\{ S\in\mcl D_n : \sup_{u,v\in S} D_h(u,v) \in \left[ 2^{-s_k n}  , 2^{-s_{k-1} n} \right] \right\} ,
\quad\forall n \in \BB N , 
\quad \forall k = 1,\dots , K .
\eqe
We also define
\eqb
\mcl D_{n,K+1} :=  \left\{ S\in\mcl D_n : \sup_{u,v\in S} D_h(u,v) < 2^{-s_K n} \right\} .
\eqe

By Lemma~\ref{lem-eucl-square-count}, it holds except on an event of probability decaying like some $\zeta$-dependent positive power of $2^{-n}$ that
\eqb \label{eqn-use-eucl-square-count}
\bigcup_{k=1}^{K+1} \mcl D_{n,k} = \mcl D_n \quad \text{and} \quad
 \# \mcl D_{n,k} \leq 2^{ \left(2 -  \frac{(\xi Q - s_k)^2}{2\xi^2} + \zeta  \right) n } , 
\quad \forall k =1,\dots, K    .
\eqe  
By the Borel-Cantelli lemma, a.s.~\eqref{eqn-use-eucl-square-count} holds for each sufficiently large $n\in\BB N$. 
Henceforth assume that this is the case.
The rest of the argument is purely deterministic. 
\medskip

\noindent\textit{Step 2: reducing to a bound for sums over squares.}
Let $X\subset [0,1]^2$ be a (possibly random) Borel set. If $ \dim_{\mcl H}^0 X \geq 2-\gamma^2/2$, then the right side of~\eqref{upper-bound-quantum} is equal to $d_\gamma$.
Hence we can assume without loss of generality that $\dim_{\mcl H}^0 X < 2-\gamma^2/2$. 
Let $x\in \left( \dim_{\mcl H}^0 X , 2-\gamma^2/2 \right) $ and let
\eqb \label{eqn-eucl-square-ell}
  \ell > \frac{x}{\xi\left( Q - \sqrt{4-2x} \right) } .
\eqe
We will show that $\dim_{\mcl H}^{\gamma}X \leq \ell$. 
Since $x$ can be made as close as we like to $\dim_{\mcl H}^0 X$ and $\ell$ can be made as close as we like to the right side of~\eqref{eqn-eucl-square-ell}, this will imply~\eqref{upper-bound-quantum}. 

For each $n_0 \in \BB N$, there exists a countable collection $\mcl S = \mcl S(n_0)$ of dyadic squares contained in $[0,1]^2$ such that
\eqb \label{eqn-eucl-cover}
X\subset \bigcup_{S \in \mcl S} S , \quad \sum_{S\in\mcl S} |S|^x \leq 1 , \quad \text{and} \quad \max_{S\in\mcl S} |S|\leq 2^{-n_0} .
\eqe
We assume that $n_0$ is large enough that~\eqref{eqn-use-eucl-square-count} holds for each $n\geq n_0$. 
We will show that
\eqb \label{eqn-quantum-diam-sum}
\sum_{S\in\mcl S}\left(   \sup_{u,v\in S} D_h(u,v) \right)^\ell  \rta 0 ,\quad\text{as $n_0\rta\infty $} ,
\eqe
which will imply that $\dim_{\mcl H}^{\gamma} X \leq \ell$, as required. 
\medskip

\noindent\textit{Step 3: bound for the ``worst-case" collection of squares.}
To prove~\eqref{eqn-quantum-diam-sum}, we first observe that $\sum_{n=n_0}^\infty  2^{- x n} \#(\mcl S\cap \mcl D_n)  \leq 1$, which implies that
\eqb \label{eqn-eucl-square-count'}
\#(\mcl S\cap \mcl D_n)  \leq  2^{ x n}  , \quad\forall n\geq n_0 .
\eqe
By~\eqref{eqn-eucl-square-count'}, the largest possible value of the sum~\eqref{eqn-quantum-diam-sum} occurs when for each $n \geq n_0$, the set $\mcl S\cap \mcl D_n$ consists of the $\lfloor 2^{x n} \rfloor$ squares of $\mcl D_n$ with the largest $D_h$-diameters. 

We will now upper-bound the sum of the $\ell$th powers of the $\lfloor 2^{x n} \rfloor$ largest $D_h$-diameters of squares in $\mcl D_n$. 
Recalling the exponent appearing on the right side of~\eqref{eqn-use-eucl-square-count}, let $k(x)$ be the smallest integer in $[0,K]$ for which 
\eqb \label{eqn-s-choice}
    2 -  \frac{(\xi Q - s_{k(x)} )^2}{2\xi^2} + \zeta    > x   , \quad \text{so that} \quad s_{k(x)} = \xi \left( Q   - \sqrt{4 - 2x}  \right) + o_\zeta(1)  .
\eqe 
Since $x < 2$, we have $s_{k(x)} < \xi Q$ (and hence $k(x) < K$) provided $\zeta \in (0,1)$ is chosen to be sufficiently small. 
By~\eqref{eqn-use-eucl-square-count}, each of the $\lfloor 2^{x n}\rfloor$ squares in $\mcl D_n$ with the largest $D_h$-diameters belongs to $\mcl D_{n,k}$ for some $k  \leq  k(x) $. 
Using~\eqref{eqn-use-eucl-square-count} in the second inequality, we thus have
\eqb \label{eqn-wc-lower-sum}
\sum_{S\in \mcl S\cap \mcl D_n} \left(   \sup_{u,v\in S} D_h(u,v) \right)^\ell 
\leq \sum_{k=1}^{k(x)} 2^{- \ell s_{k-1} n } \#  \mcl D_{n,k}  
\leq \sum_{k=1}^{k(x)} 2^{ \left(2 -  \frac{(\xi Q - s_k)^2}{2\xi^2} -   \ell s_k   + o_\zeta(1)  \right)  n }  .
\eqe
\medskip

\noindent\textit{Step 4: bounding the right side of~\eqref{eqn-wc-lower-sum}.}
We want to argue that the right side of~\eqref{eqn-wc-lower-sum} is dominated by the summand for $k = k(x)$. 
To make sure that this is the case we need to be careful about our choice of parameters. 
The function
\eqbn
s \mapsto  2 -  \frac{(\xi Q - s)^2}{2\xi^2} -   \ell s
\eqen
attains a unique maximum value at $s = \xi ( Q - \ell \xi)$, and is increasing for values of $s$ smaller than this maximum. 
Since we have assumed that $x  <2-\gamma^2/2$, if $\ell$ is chosen to be sufficiently close to the right side of~\eqref{eqn-eucl-square-ell} then
\eqbn
\xi \left( Q   - \sqrt{4 - 2x}  \right) < \xi(Q-\ell\xi) .
\eqen
By~\eqref{eqn-s-choice}, this means that $s_{k(x)}  < \xi(Q-\ell \xi)$ for each small enough choice of $\zeta > 0$. 
For such a choice of $\zeta$, the right side of~\eqref{eqn-wc-lower-sum} is bounded above by a $\zeta$-dependent constant times the summand for $k = k(x)$.
By applying~\eqref{eqn-wc-lower-sum} and then~\eqref{eqn-s-choice}, we obtain
\eqb \label{eqn-wc-lower-term}
\sum_{S\in \mcl S \cap \mcl D_n} \left(   \sup_{u,v\in S} D_h(u,v) \right)^\ell 
\leq 2^{ \left( x -   \ell s_{k(x)}   + o_\zeta(1)  \right) n  } 
= 2^{\left( x - \ell \xi \left( Q   - \sqrt{4 - 2x}  \right) + o_\zeta(1) \right) n } .
\eqe
By~\eqref{eqn-eucl-square-ell}, the power of $2^n$ on the right side of~\eqref{eqn-wc-lower-term} is negative for a small enough choice of $\zeta$. 
Hence we can sum~\eqref{eqn-wc-lower-term} over all $n \geq n_0$ and use~\eqref{eqn-eucl-square-ell} to get
\eqb
\sum_{S\in\mcl S} \left(   \sup_{u,v\in S} D_h(u,v) \right)^\ell 
\leq \sum_{n=n_0}^\infty 2^{\left( x - \ell \xi \left( Q   - \sqrt{4 - 2x}  \right) + o_\zeta(1) \right) n } 
\rta 0 ,\quad \text{as $n_0\rta \infty$.}
\eqe
Therefore,~\eqref{eqn-quantum-diam-sum} holds.
\end{proof}

\subsection{Upper bound for Euclidean dimension}
\label{sec-wc-euclidean}

The proof of~\eqref{upper-bound-euclidean} is very similar to the proof of~\eqref{upper-bound-quantum} in the preceding subsection.
We will define for each $m\in\BB N$ a collection $\mcl R_m$ of dyadic squares in $[0,1]^2$ which all have ``$\gamma$-LQG size" approximately $2^{-m}$ in an appropriate sense (but random Euclidean sizes), which will play the same role as $\mcl D_n$ in the proof of~\eqref{upper-bound-quantum}. 

To prove~\eqref{upper-bound-euclidean}, we will consider a Borel set $X \subset [0,1]^2$, a number $\ell >  \dim_{\mcl H}^{\gamma} X$, and a number $x$ greater than the right side of~\eqref{upper-bound-euclidean}. 
We will cover $X$ by a collection $\mcl A$ of small sets for which $\sum_{A\in\mcl A} (\sup_{u,v \in A} D_h(u,v) )^\ell \leq 1$.
The definition of $\mcl R_m$ (see in particular Lemma~\ref{lem-quantum-tiling-contain}) will allow us to find for each $A\in\mcl A$ whose quantum diameter is in $[2^{-m-1}, 2^{-m}]$ a square $S_A \in \mcl R_m$ whose side length is comparable to the Euclidean diameter of $A$.
We will then set $\mcl S  =\{S_A : A\in\mcl A\}$ and show that $\sum_{S\in\mcl S} |S|^x $ is small. 
To do this, we will consider the ``worst-case scenario" where for each large enough $m$, the squares in $\mcl S\cap \mcl R_m$ are the $\lfloor 2^{\ell m+1} \rfloor$ squares of $\mcl R_m$ with the largest Euclidean side lengths. 
Lemma~\ref{lem-quantum-square-count} just below will allow us to show that the sum of the $D_h$-diameters of these $\lfloor 2^{\ell m+1} \rfloor$ largest squares is bounded above by a positive power of $2^{-m}$, and we will conclude by summing over $m$. 

We now define our random dyadic tiling. 
For a dyadic square $S$, we write $S'$ for its dyadic parent, i.e., the unique dyadic square of side length $2|S|$ which contains $S$. 
For $m \in \BB N$, we define the dyadic tiling $\mcl R_m$ of $[0,1]^2$ to be the set of dyadic squares $S\subset [0,1]^2$ such that
\eqb \label{eqn-quantum-tiling}
D_h\left( S , \bdy B_{|S|}(S) \right) \leq 2^{-m} \quad \text{and} \quad 
D_h\left( S' , \bdy B_{|S'|}(S') \right) > 2^{-m}  . 
\eqe
Then $\mcl R_m$ is a tiling of $[0,1]^2$, i.e., the squares in $\mcl R_m$ cover $[0,1]^2$ and intersect only along their boundaries. 
The reason for our particular definition of $\mcl R_m$ is to make the following lemma true. 

\begin{lem} \label{lem-quantum-tiling-contain}
If $S\in \mcl R_m$ and $A\subset \BB C$ is a set of $D_h$-diameter at most $2^{-m}$ which intersects $S$, then $A\subset B_{4|S|}(S )$.
\end{lem}
\begin{proof}
The set $A$ must intersect $S'$, so since $ D_h\left( S' , \BB C\setminus  B_{|S'|}(S') \right)  = D_h\left( S' , \bdy B_{|S'|}(S') \right) > 2^{-m}$ we must have $A\subset B_{|S'|}(S') \subset B_{4|S|}(S)$.
\end{proof}

\begin{lem} \label{lem-quantum-square-count}
For each $t \in ( 0  ,  1/(\xi Q))     $ and each $\zeta \in (0,1)$, there exists $\alpha =\alpha(t,\zeta) > 0$ such that for each $m \in\BB N$, it holds with probability at least $1-O_m(2^{-\alpha m})$ that
\eqb \label{eqn-quantum-square-count}
\#\left\{ S\in\mcl R_m : |S| > 2^{-t m} \right\} \leq 2^{ \left( 2 t - \frac{(1 - t\xi Q)^2}{2\xi^2 t}      + \zeta \right) m } .
\eqe 
In particular, if $t < \xi^{-1}(Q+2)^{-1}$, then there exists $\alpha' = \alpha'(s) >0$ such that for each $m\in \BB N$, it holds with probability at least $1-O_m(2^{-\alpha m})$ that the set in~\eqref{eqn-quantum-square-count} is empty. 
\end{lem}
\begin{proof}
The same proof as in~\cite[Lemma 3.21]{lqg-metric-estimates} (applied with $\BB r  = 1$, $\ep = 2^{-\lfloor t m \rfloor}$, and $s = 1/t$) shows that for $S \in \mcl D_{\lfloor t m \rfloor}$, 
\eqb \label{eqn-quantum-square-prob}
\BB P\left[ S\in \mcl R_m \right]
\leq \BB P\left[ D_h\left( S , \bdy B_{|S|}(S) \right) \leq 2^{-m} \right]
\leq 2^{-m\left( \frac{ (1 - t \xi Q)^2}{2\xi^2 t} + o_m(1) \right)} .    
\eqe
Using~\eqref{eqn-quantum-square-prob} and Axiom~\ref{item-metric-coord}, it is easily seen that~\eqref{eqn-quantum-square-prob} still holds if $n \in [0, t m ]_{\BB Z}$ and $S\in\mcl D_n$.  
Summing~\eqref{eqn-quantum-square-prob} over all $ n= 0,\dots, \lfloor t m\rfloor$ and all $S\in\mcl D_n$ gives
\eqb
\BB E\left[ \#\left\{ S\in\mcl R_m : |S| > 2^{-t m} \right\} \right] \leq  2^{  \left( 2 t - \frac{(1 - t\xi Q)^2}{2\xi^2 t} + o_m(1) \right) m} .
\eqe
The bound~\eqref{eqn-quantum-square-count} follows from this and Markov's inequality.
The last statement follows since if $t < \xi^{-1}(Q+2)^{-1}$, then the exponent on the right side of~\eqref{eqn-quantum-square-count} is negative for a small enough choice of $\zeta$ (depending on $t$). 
\end{proof}

\begin{proof}[Proof of~\eqref{upper-bound-euclidean}.]
\noindent\textit{Step 1: setup.}
Fix a small parameter $\zeta \in (0,1)$ (which we will eventually send to zero) and a partition of $[\xi^{-1}(Q+2)^{-1} -\zeta , 1/(\xi Q)]$ 
\eqb \label{eqn-partition'}
\xi^{-1}(Q+2)^{-1} -\zeta = t_0 < \dots < t_K = 1/(\xi Q) \quad \text{with} \quad \max_{k=1,\dots,K} (t_k  -t_{k-1}) \leq \zeta .
\eqe 
We can arrange that the $t_k$'s depend only on $\zeta$. 
Throughout the proof, $o_\zeta(1)$ denotes a deterministic quantity which depends only on $\zeta$ and $\gamma$ and which converges to zero as $\zeta \rta 0$. 

For $m\in\BB N$, define the random dyadic tiling $\mcl R_m$ as in~\eqref{eqn-quantum-tiling}. 
We partition the set of dyadic squares in $\mcl R_{m}$ according to their quantum diameters by defining
\eqb \label{eqn-eucl-square-partition}
\mcl R_{m , k}
:=  \left\{ S\in\mcl R_{m} : |S| \in \left[ 2^{-t_k m}  , 2^{-t_{k-1} m} \right] \right\} ,
\quad\forall m \in \BB N, 
\quad \forall k = 1,\dots , K .
\eqe
We also define
\eqbn
\mcl R_{m , K+1}
:=  \left\{ S\in\mcl R_{m} : |S|  <  2^{-t_K m} \right\}  .
\eqen

By Lemma~\ref{lem-eucl-square-count}, it holds except on an event of probability decaying like some $\zeta$-dependent positive power of $2^{-m}$ that
\eqb \label{eqn-use-quantum-square-count} 
\bigcup_{k=1}^{K+1} \mcl R_{m,k} =\mcl R_m \quad \text{and} \quad
 \# \mcl R_{m ,k} \leq 2^{ \left(2 t_k -  \frac{(1  - t_k \xi Q)^2}{2\xi^2 t_k} + \zeta  \right) m } , 
\quad \forall k =1,\dots, K   .
\eqe 
By the Borel-Cantelli lemma, it is a.s.\ the case that~\eqref{eqn-use-quantum-square-count} holds for each sufficiently large $m \in\BB N$. 
Henceforth assume that this is the case.
The rest of the argument is purely deterministic. 
\medskip

\noindent\textit{Step 2: reducing to a bound for sums over squares.}
Let $X\subset [0,1]^2$ be a (possibly random) Borel set.  
We can assume without loss of generality that $\dim_{\mcl H}^{\gamma}X < 2/(\xi Q)$, for otherwise the right side of~\eqref{upper-bound-euclidean} equals 2, so~\eqref{upper-bound-euclidean} holds vacuously.
Let $\ell \in \left(  \dim_{\mcl H}^{\gamma} X , 2/(\xi Q) \right)$ and let $x  \in (0,2)$ be chosen so that
\eqb \label{eqn-quantum-square-x}
x > \ell \xi\left( Q - \ell  \xi  + \sqrt{4 - 2Q \ell \xi  + \ell^2 \xi^2} \right) .
\eqe
We will show that $\dim_{\mcl H}^0 X\leq x$. 
Since $\ell$ can be made arbitrarily close to $\dim_{\mcl H}^0 X$ and $x \in (0,2)$ was chosen arbitrarily subject to the constraint~\eqref{eqn-quantum-square-x}, this implies~\eqref{upper-bound-euclidean}. 

For each $m_0 \in \BB N$, there exists a countable collection $\mcl A = \mcl A(m_0)$ of Borel sets $A$ which intersect $[0,1]^2$ (which we can take to be $D_h$-metric balls) such that
\eqb \label{eqn-quantum-cover}
X\subset \bigcup_{A \in \mcl A } A , \quad \sum_{A\in\mcl A} \left( \sup_{u,v\in A} D_h(u,v) \right)^\ell \leq 1 , \quad \text{and} \quad \sup_{A\in\mcl A} \sup_{u,v\in A} D_h(u,v) \leq 2^{-m_0} .
\eqe
We assume that $m_0$ is large enough that~\eqref{eqn-use-quantum-square-count} holds for each $m\geq m_0$. 

For each $A \in \mcl A$, let $m_A \geq m_0$ be chosen so that $2^{-m_A - 1} \leq \sup_{u,v\in A} D_h(u,v)  \leq 2^{-m_A}$ and let $S_A \in \mcl R_{m_A}$ be chosen so that $A\cap S_A\not=\emptyset$. By Lemma~\ref{lem-quantum-tiling-contain}, $A\subset B_{4|S_A|}(S_A)$. 
Let
\eqbn
\mcl S := \left\{S_A : A\in\mcl A \right\} .
\eqen
Then the Euclidean neighborhoods $B_{4|S|}(S)$ for $S\in\mcl S$ cover $X$.
The Euclidean diameter of each of these neighborhoods is at most a universal constant times $|S|$. 
We claim that if $\zeta$ is chosen to be sufficiently small (depending on $\ell$ and $x$), then
\eqb \label{eqn-eucl-diam-sum}
\sum_{S\in\mcl S} |S|^x  \rta 0 ,\quad\text{as $m_0\rta\infty $} .
\eqe
This will imply that $\dim_{\mcl H}^0 X \leq x$. 
\medskip

\noindent\textit{Step 3: bound for the ``worst-case" collection of squares.}
By the above definition of $S$, if $S\in\mcl S\cap \mcl R_m$ for some $m\geq m_0$ then $S = S_A$ for some $A\in\mcl A$ with $2^{-m_A - 1} \leq \sup_{u,v\in A} D_h(u,v)  \leq 2^{-m_A}$. 
By~\eqref{eqn-quantum-cover}, 
\eqbn
\sum_{m=m_0}^\infty 2^{-\ell m-1} \#(\mcl S\cap \mcl R_m) \leq  \sum_{A\in\mcl A} \left( \sup_{u,v\in A} D_h(u,v) \right)^\ell   \leq 1 .
\eqen 
Therefore,
\eqb \label{eqn-quantum-square-count'}
\#(\mcl S\cap \mcl R_m) \leq  2^{ \ell m + 1 }  , \quad\forall m\geq m_0 .
\eqe
By~\eqref{eqn-quantum-square-count'}, the largest possible value of the sum~\eqref{eqn-eucl-diam-sum} occurs when for each $m \geq m_0$, the set $\mcl S\cap \mcl R_m$ consists of the $\lfloor 2^{\ell m+1} \rfloor$ squares of $\mcl R_m$ with the largest Euclidean side lengths. 

We will now upper-bound the sum of the $x$th powers of the $\lfloor 2^{\ell m+1} \rfloor$ largest side lengths of squares in $\mcl R_{m}$. 
To this end, let $k(\ell)$ be the smallest integer in $[0,K]$ for which 
\eqb \label{eqn-t-choice}
    2 t_{k(\ell)} -  \frac{(  1  - t_{k(\ell)} \xi Q )^2}{2\xi^2 t_{k(\ell)}} + \zeta    > \ell  
    ,\quad \text{so that} \quad t_{k(\ell)} = \frac{1}{\xi\left( Q -  \ell \xi + \sqrt{4 - 2 \ell \xi Q + \ell^2\xi^2 } \right)} +o_\zeta(1) .
\eqe
Since $\ell < 2/(\xi Q)$, we have $t_{k(\ell)} < 1/(\xi Q)$ for each small enough $\zeta \in (0,1)$.
By~\eqref{eqn-use-quantum-square-count}, the $\lfloor 2^{\ell m + 1}\rfloor$ squares in $\mcl R_{m}$ with the largest side lengths each belong to $\mcl R_{m ,k}$ for some $k  \leq  k(\ell) $. 
Therefore, we can apply~\eqref{eqn-use-quantum-square-count} to obtain
\eqb \label{eqn-wc-upper-sum}
\sum_{S\in \mcl S\cap \mcl R_m} |S|^x
\leq \sum_{k=1}^{k(\ell)} 2^{- x t_{k-1} m } \# \mcl R_{m ,k}
\leq \sum_{k=1}^{k(\ell)} 2^{ \left(2 t_k -  \frac{(1  - t_k \xi Q)^2}{2\xi^2 t_k} -   x t_k   + o_\zeta(1)  \right)  m }  .
\eqe
\medskip

\noindent\textit{Step 4: bounding the right side of~\eqref{eqn-wc-upper-sum}.}
We want to argue that the right side of~\eqref{eqn-wc-lower-sum} is dominated by the summand for $k = k(\ell)$. 
Indeed, the function 
\eqbn
t \mapsto  2 t -  \frac{(1 - t \xi Q  )^2}{2\xi^2 t} -   x t
\eqen
is increasing on $(0,1/(\xi Q))$.  
Hence, for small enough $\zeta$, the right side of~\eqref{eqn-wc-lower-sum} is bounded above by a $\zeta$-dependent constant times the summand for $k = k(\ell)$.
By applying~\eqref{eqn-wc-upper-sum} and then~\eqref{eqn-t-choice}, we obtain
\eqb \label{eqn-wc-upper-term}
\sum_{S\in \mcl S\cap \mcl R_m} |S|^x
\leq 2^{ \left( \ell  -   x t_{k(\ell)}   + o_\zeta(1)  \right) m  }  
\leq 2^{\left( \ell - \frac{x}{\xi\left( Q - \ell  \xi + \sqrt{4 -  2 \ell \xi Q + \ell^2\xi^2 } \right)} + o_\zeta(1) \right) m} .
\eqe
By our choice of $x$ in~\eqref{eqn-quantum-square-x}, the power of $2^m$ on the right side of~\eqref{eqn-wc-upper-term} is negative for each small enough $\zeta \in (0,1)$. 
Hence for small enough $\zeta$ we can sum~\eqref{eqn-wc-lower-term} over all $m \geq m_0$ and apply~\eqref{eqn-quantum-square-x} to get
\eqb
\sum_{S\in\mcl S} |S|^x
\leq \sum_{m=m_0}^\infty 2^{\left(   \ell - \frac{x}{\xi\left( Q -  \ell\xi  + \sqrt{4 - 2 \ell \xi Q + \ell^2\xi^2 } \right)}     + o_\zeta(1) \right) m  } 
\rta 0 ,\quad \text{as $m_0\rta \infty$} .
\eqe
Therefore,~\eqref{eqn-eucl-diam-sum} holds.
\end{proof}

\bibliography{cibib}
\bibliographystyle{hmralphaabbrv}

\end{document}